\documentclass[a4paper]{amsart}
\usepackage[margin=1in]{geometry}
\usepackage{booktabs}
\usepackage{makecell}
\usepackage{thmtools}
\usepackage{mathtools}
\usepackage[backref=page,colorlinks=true,citecolor=green!70!black]{hyperref}
\usepackage[export]{adjustbox}
\usepackage[capitalize]{cleveref}
\usepackage{tikz-cd}
\usepackage{enumitem}
\usepackage[charter]{mathdesign}
\usepackage{todonotes}
\presetkeys{todonotes}{inline}{}
\usepackage{verbatim}

\definecolor{darkred}{rgb}{0.7,0,0} \definecolor{darkblue}{rgb}{0,0,0.7} \newcommand{\darkblue}{\color{darkblue}} \declaretheorem[parent=section]{theorem}
\declaretheorem[sibling=theorem, style=definition]{definition}
\declaretheorem[sibling=theorem]{lemma}
\declaretheorem[sibling=theorem]{proposition}
\declaretheorem[sibling=theorem]{corollary}
\declaretheorem[sibling=theorem, style=remark]{remark}

\DeclareMathOperator{\depth}{depth}
\DeclareMathOperator{\pal}{Pal}
\newcommand{\defn}[1]{\textsl{\darkblue #1}} 

\usetikzlibrary{automata}

\title{On the language of palindromic reduced words in a Coxeter group}
\author{Asilata Bapat}
\address[Asilata Bapat]{Mathematical Sciences Institute, Australian National University}
\email{asilata.bapat@anu.edu.au}
\begin{document}
\begin{abstract}
  We prove that the language of palindromic reduced words in a Coxeter system is regular if and only if the Coxeter group is a direct product of a finite Coxeter group with finitely many copies of the infinite dihedral group.
\end{abstract}
\maketitle

\section{Introduction}
Let \((W,S)\) be a Coxeter system.
Then \(W\) is a group generated by the finite set \(S\).
Each element of \(S\) has order two, and the only relations between them are of dihedral type (see~\cref{sec:coxeter-basics}).
A \emph{reflection} is any element of \(W\) that is conjugate to some element of \(S\).
Since all generators have order two, any element of \(W\) is a product of generators.
Moreover, the inverse of any such product is simply its reverse.
Thus if \(t = w s w^{-1}\) is a reflection for some \(s \in S\) and some \(w \in W\), then \(t\) is a palindromic word in the generators \(S\).

In fact, any reflection has a palindromic \emph{reduced} word (see, e.g.~\cite[Lemma 1.4]{dye:87} as well as~\cite[Exercise 1.10]{bjo.bre:05}).
That is, a palindromic word with minimal possible length among all words in the generating set \(S\).
Conversely, any non-trivial word that has a palindromic reduced word as \(w s w^{-1}\) is clearly a reflection.

The aim of this article is to answer a question of Biagioli--Hohlweg--Sasso~\cite{bia.hoh.sas:26}, which asks whether the language of palindromic reduced words in the generators \(S\) is regular.
Let \(\pal\) be the language of all palindromic reduced words in \((W,S)\).
We regard \(\pal\) as a formal language in \(S\).
That is, \(\pal\) is a subset of \(S^{\ast}\), which is the free monoid on the elements of \(S\).
Recall that a language is called~\emph{regular} if there is a finite automaton that recognises it.
If \(W\) is a finite group, then \(\pal\) is finite and hence regular.
We prove the following theorem that answers the question fully for any Coxeter system.

\begin{theorem}[Main theorem]\label{thm:all-coxeter-groups-main}
  Let \((W,S)\) be a Coxeter system.
  Then the language \(\pal\) of palindromic reduced words in \((W,S)\) is regular if and only if \(W\) is the direct product of a finite Coxeter system with finitely many copies of the infinite dihedral group.
\end{theorem}

It was already known to the experts that there exist Coxeter groups for which \(\pal\) is irregular~\cite{dye:26}, a fact that has motivated some of the results of~\cite{edg:09}.
However, as far as we are aware, our main theorem is the first complete answer to this question.

We give some context and motivation.
Geodesic languages for Coxeter groups have been studied extensively.
These are languages consisting of reduced words such that each element of the group is represented by at least one word in the language.
For example, the language consisting of all reduced words in \(S^{\ast}\) is a geodesic language.
This language is regular in any Coxeter group (Brink--Howlett~\cite{bri.how:93}).
Similarly, the language of lexicographically minimal reduced words is also a geodesic language.
This language is also regular, recognised by a variant of Brink and Howlett's automaton~\cite{bri.how:93}.
It follows that the generating series that encodes the number of reduced words of each length in a fixed Coxeter system is a rational function (see, e.g.~the discussion around~\cite[Theorem 4.9.1 and Corollary A4.1.3]{bjo.bre:05}).

Stembridge asked in 1997 (see~\cite[Problem 1.1]{bre:24}) whether the generating series that encodes the reflections of every length in a Coxeter group is rational.
This question was answered positively in~\cite{man:99}.
Producing an automaton that recognises exactly one reduced word per reflection would give a direct proof of this fact.
The language \(\pal\) is the natural candidate geodesic language for the set of all reflections in \(W\).
It is a natural question to ask whether (or when) this language is regular: had it been regular, we may have hoped for a Brink--Howlett type construction for an automaton to recognise the lexicographically minimal palindromic reduced words.

Let us also mention some related results.
Theorem 1.1 of~\cite{bia.hoh.sas:26} proves that the language of all \emph{reflection-prefixes} (that is, words of the form \(ws\) where \(wsw^{-1}\) is a reduced word) is regular for any Coxeter system.
It follows that the generating series for the language \(\pal\) is rational.
For the finite Coxeter groups, for which we know that \(\pal\) is regular, an exhaustive enumeration of \(\pal\) has been done by Milićević~\cite{mil:25}.
The only other regular case (up to products) in~\cref{thm:all-coxeter-groups-main} is the infinite dihedral group; an enumeration of \(\pal\) in this case is easy (see~\cref{prop:affine-a1} and its proof).
Finally, we draw the reader's attention to~\cref{prop:bootstrap-cardinality}, a curious byproduct of our proof strategy that gives an alternate proof of the classification of all finite Coxeter groups (see~\cref{rem:new-proof-classification}).

\subsection*{Acknowledgements}
I thank Hoel Queffelec and the students in my class on Coxeter groups for inspiration and conversations about~\cite[Exercise 1.10]{bjo.bre:05}, which pointed me towards this problem.
I am grateful to Christophe Hohlweg for a careful reading of an earlier draft of this document and insightful questions that led to a significant strengthening of the main theorem.
I am indebted to Matthew Dyer for pointing out an error in the earlier version of the proof of~\cref{prop:affine-a1}, and for pointing out connections to previous work.
I thank Anand Deopurkar, Yixuan Li, and Anthony Licata for clarifying conversations.
I acknowledge support from the France-Australia Mathematical Sciences and Interactions ANU-CNRS International Laboratory, as well as the ARC grants DE240100447 and DP240101084.

\subsection*{AI Statement}
The author did not use AI to assist with any aspect of this document.

\section{Basic theory of Coxeter groups}\label{sec:coxeter-basics}
We briefly recall the standard facts we need about Coxeter groups, omitting most proofs.
We refer the reader to the standard reference books~\cite{bjo.bre:05,hum:90} for more details.
\begin{definition}
  A \defn{Coxeter system} \((W,S)\) consists of a group \(W\) and a finite generating set \(S\).
  The relations of \(W\) are encoded by a function \(m \colon S \times S \to \{2, \ldots, \} \cup \{\infty\}\) with the property that \(m(s,s') = 2\) if and only if \(s = s'\).
  Then the relations in \(W\) are
  \((ss')^{m(s,s')} = 1\)
  for any \(s, s' \in S\) such that \(m(s,s') \neq \infty\).
\end{definition}
It is customary to encode the information of the function \(m\) as a labelled undirected graph, called the \emph{Coxeter graph} and often denoted \(\Gamma\).
(See~\cref{tab:finite-coxeter} later in this paper for several examples.)
We use the following conventions, which match~\cite[Section 1.1]{bjo.bre:05} as well as~\cite[Section 2.1]{hum:90}.
\begin{enumerate}
\item The vertex set of \(\Gamma\) is in bijection with the elements of \(S\).
\item We draw an edge between \(s\) and \(s'\) if \(m(s,s') \geq 3\).
\item If \(m(s,s') \geq 4\), then we label the edge with \(m(s,s')\).
\end{enumerate}
If \(m(s,s') \leq 3\) for all \(s,s' \in S\), then we call the corresponding Coxeter system or Coxeter graph \emph{simply-laced}.
Conversely, a Coxeter graph \(\Gamma\) specifies an associated Coxeter system \((W_{\Gamma}, S_{\Gamma})\).

Henceforth in this section, fix a Coxeter system \((W,S)\).
\begin{definition}
  The elements of \(S\) are called the \defn{simple reflections} of \(W\).
  A~\defn{reflection} of \(W\) is any conjugate of an element of \(S\).
  That is, the reflections are \(\{w s w^{-1} \mid s \in S, w \in W\}\).
\end{definition}
We note again that because elements of \(S\) have order two, every element \(w \in W\) can simply be expressed as a product of elements of \(S\).
Indeed, the inverse of any product of elements of \(S\) is simply the reversed string.

Recall that \(S^{\ast}\) is the free monoid on the generators \(S\).
A \defn{word} in \((W,S)\) is an element of \(S^{\ast}\).
There is a natural map \(S^\ast \to W\), taking a word to its product in \(W\).
For clarity, we write words with boldface letters (e.g.~\(\mathbf{w}\)), and their images in \(W\) with the corresponding non-boldface letter (e.g.~\(w\)).
In this case we say that \(\mathbf{w}\) represents \(w\).
We denote by \(\overline{\mathbf{w}}\) the reverse of a word \(\mathbf{w}\).
We denote by \(|\mathbf{w}|\) the number of letters in \(\mathbf{w}\).
\begin{definition}
  Let \(w \in W\).
  Then the \defn{length} of \(w\), denoted \(\ell(w)\), is the minimum of the quantity \(|\mathbf{w}|\) over all words \(\mathbf{w}\) that represent \(w\).
\end{definition}
If \(\mathbf{w}\) is any word representing \(w\), we clearly have \(\ell(w) \leq |\mathbf{w}|\).
We say that \(\mathbf{w}\) is \defn{reduced} if \(|\mathbf{w}| = \ell(w)\).
The length function has a number of useful and pleasing properties, whose discussion we omit.

The Coxeter system \((W,S)\) has a distinguished real representation \(V_W\) called the~\defn{standard geometric representation}, defined as follows.
The underlying vector space of \(V_W\) is the free real vector space generated by elements \(\{\alpha_s \mid s \in S\}\).
Fix a symmetric bilinear form on \(V_W\), defined on the basis as
\begin{equation}\label{eq:bilinear}
  (\alpha_s, \alpha_{s'}) = - \cos\left(\frac{\pi}{m(s,s')}\right),
\end{equation}
where we set \(\cos(\pi/\infty) = 1\).
Now the action of a generator \(s \in S \subset W\) on an element \(v \in V_W\) is defined as
\[s(v) = v - 2 (\alpha_s, v) \alpha_s.\]
\begin{proposition}\label{prop:geom-rep-faithful}
  The linear maps \(s \colon V_W \to V_W\) defined by
  \(s(v) = v - 2 (\alpha_s, v) \alpha_s\)
  extend to a faithful linear representation \(W \to GL(V_W)\), called the standard geometric representation of \((W,S)\).
\end{proposition}
\begin{definition}
  The basis elements \(\alpha_t \in V_W\) for \(t \in S\) are called \defn{simple roots}.
  More generally, a \defn{root} is any element of \(V_W\) that can be expressed as \(w(\alpha_t)\) for some \(t \in S\) and some \(w \in W\).
\end{definition}
We call the vector space \(V_W\) the \defn{root space}.
A \defn{positive root} is a non-negative linear combination of the basis elements \(\alpha_s\).
A \defn{negative root} it is a non-positive linear combination of the basis elements \(\alpha_s\).
It is known that every root is either a positive root or a negative root (see, e.g.~\cite[Proposition 4.2.5]{bjo.bre:05}).
If \(\alpha = w(\alpha_t)\) is a root for some \(t \in S\), we denote by \(s_{\alpha}\) the reflection in the root \(\alpha\), which is equal to
  \[s_{\alpha} = w t w^{-1}.\]

\begin{definition}\label{def:depth}
  Let \(\alpha\) be a root.
  Suppose that \(\mathbf{w}\) is a minimal length word such that \(\alpha = w(\alpha_j)\) for some simple root \(\alpha_j\).
  Then \defn{depth} of \(\alpha\) is defined to be the quantity \(|\mathbf{w}| + 1\).
\end{definition}

The following statement is~\cite[Lemma~4.6.2]{bjo.bre:05}, and explains how the depth of a root changes when it is acted upon by a simple reflection.
\begin{proposition}\label{prop:depthchange}
  Let \(s \in S\) and \(\alpha \neq \alpha_s\) be a positive root.
  Then
  \[
    \depth(s(\alpha)) =
    \begin{cases}
      \depth(\alpha) - 1,&(\alpha, \alpha_s) > 0,\\
      \depth(\alpha),&(\alpha, \alpha_s) = 0,\\
      \depth(\alpha) + 1,&(\alpha, \alpha_s) < 0.
    \end{cases}
  \]
\end{proposition}
The \defn{dual numbers game} (see, e.g.~\cite[Section 4.6]{bjo.bre:05}) is a particularly convenient method to apply letters in the Coxeter generators to vectors in the root space.
For this method, we represent a vector \(\beta = \sum_{s \in S} c_s \alpha_s\) in the root space by writing the coefficients directly on the Coxeter graph.
Specifically, we label the graph by writing \(c_s\) at the node corresponding to \(s \in S\).
For example, the following picture represents the vector \(5 \alpha_1 + 3 \alpha_2 + 7\alpha_3 + \alpha_4\) if the nodes are labelled \(1,2,3,4\) in order from left to right.
\begin{center}
  \begin{tikzpicture}[scale=0.5, baseline=(current bounding box.center), font=\scriptsize]
    \draw (0 cm,0) -- (6 cm,0);
    \draw[fill=white] (0 cm, 0 cm) circle (0.4cm) node{\(5\)};
    \draw[fill=white] (2 cm, 0 cm) circle (0.4cm) node{\(3\)};
    \draw[fill=white] (4 cm, 0 cm) circle (0.4cm) node {\(7\)};
    \draw[fill=white] (6 cm, 0 cm) circle (0.4cm) node {\(1\)};
  \end{tikzpicture}
\end{center}
Now if \(s_i \in S\), the coefficient that differs between \(\beta\) and \(s_i(\beta)\) is the coefficient of \(\alpha_{s_i}\) itself, which changes as follows.
\begin{enumerate}
\item Since \(s_i(\alpha_i) = - \alpha_i\), we first negate \(c_{s_i}\).
\item Whenever \(j\) is a vertex connected to \(i\) in the Coxeter graph, we have \(s_i(\alpha_j) = \alpha_j -2(\alpha_i,\alpha_j)\alpha_i\).
  Thus we add a contribution of \(-2(\alpha_i,\alpha_j) c_{s_j}\).
\end{enumerate}
The overall change then is
\[c_{s_i} \mapsto - c_{s_i} - 2 \sum_j(\alpha_i,\alpha_j) c_{s_j},\]
where \(j\) ranges over all nodes connected to \(i\) in the Coxeter graph.
In particular if the edge label is \(m(s_i,s_j) = 3\), then we have \(-2(\alpha_i,\alpha_j) = 1\).

Thus in simply-laced cases, we replace \(c_i\) by \(-c_i\) plus the sum of the labels of its neighbours.
This is easy to do visually; for instance in the example above, the element \(s_3\) acts as follows.
\[
  \begin{tikzpicture}[scale=0.5, baseline=(current bounding box.center), font=\scriptsize]
    \draw (0 cm,0) -- (6 cm,0);
    \draw[fill=white] (0 cm, 0 cm) circle (0.4cm) node{\(5\)};
    \draw[fill=white] (2 cm, 0 cm) circle (0.4cm) node{\(3\)};
    \draw[fill=white] (4 cm, 0 cm) circle (0.4cm) node {\(7\)};
    \draw[fill=white] (6 cm, 0 cm) circle (0.4cm) node {\(1\)};
  \end{tikzpicture}
  \quad
  \xmapsto{s_3}
  \quad
  \begin{tikzpicture}[scale=0.5, baseline=(current bounding box.center), font=\scriptsize]
    \draw (0 cm,0) -- (6 cm,0);
    \draw[fill=white] (0 cm, 0 cm) circle (0.4cm) node{\(5\)};
    \draw[fill=white] (2 cm, 0 cm) circle (0.4cm) node{\(3\)};
    \draw[fill=white] (4 cm, 0 cm) circle (0.4cm) node {\(-3\)};
    \draw[fill=white] (6 cm, 0 cm) circle (0.4cm) node {\(1\)};
  \end{tikzpicture}  
\]
We will need the following useful but easy corollary of~\cref{prop:depthchange}, which is also~\cite[Lemma 4.6.4]{bjo.bre:05}.
\begin{corollary}\label{cor:numbers-game-depth}
  Let \(\beta\) be an element in the root space with \(\beta = \sum_{s \in S} c_s \alpha_s\).
  Let \(s_i \in S\).
  Then the depth of \(s_i(\beta)\) is greater than the depth of \(\beta\) if and only if the label at the \(s_i\)th node increases via the corresponding move of the dual numbers game.
\end{corollary}

We also have the following well-known result; we use it extensively in this paper to check whether a given palindromic word is indeed reduced.
See, e.g.~\cite[Proposition 2.4]{dye.fis.hoh.ea:24} for a proof.
Note that in the above reference, they use a definition of depth that is off by one from the one we give in this paper.
\begin{proposition}\label{prop:depth-of-reflection}
  Let \(\alpha\) be a positive root of depth \(d\), and let \(s_{\alpha}\) be the reflection in \(\alpha\).
  Then
  \[\ell(s_{\alpha}) = 2\depth(\alpha) - 1.\]
\end{proposition}

We also recall for future reference the classification of the finite and affine Coxeter groups from e.g.~\cite[Theorem 2.7]{hum:90}.
\begin{theorem}
  Let \(\Gamma\) be a connected Coxeter graph.
  Then the associated Coxeter group \(W_{\Gamma}\) is finite if and only if \(\Gamma\) is one of the types listed in the first two columns of~\cref{tab:finite-coxeter}.
\end{theorem}

\begin{table}[ht]
  \centering
    \caption{Coxeter graphs of the finite (first two columns) and affine (second two columns) Coxeter groups.
      In the finite cases, the number of nodes in the graph is equal to the subscript in the name.
      In the affine cases, the number of nodes in the graph is equal to one more than the subscript in the name.}\label{tab:finite-coxeter}    
    \begin{tabular}{c c c c}
      \toprule
      Name & Coxeter graph & Affine version & Affine Coxeter graph\\
      \midrule
      \(A_1\) & \includegraphics[valign=m, scale=0.7]{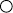} & \(\widetilde{A_1}\) & \includegraphics[valign=m, scale=0.7]{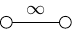}\\[1em]      
      \(A_n\) for \(n \geq 2\) & \includegraphics[valign=m, scale=0.7]{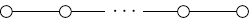} & \(\widetilde{A_n}\) & \includegraphics[valign=m, scale=0.7]{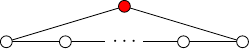}\\[1em]
      \(B_n\, (= C_n)\) for \(n \geq 2\) & \includegraphics[valign=m, scale=0.7]{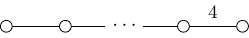} & \(\widetilde{B_n}\) & \includegraphics[valign=m, scale=0.7]{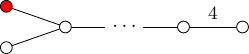}\\[1em]
      \(C_n\, (= B_n)\) for \(n \geq 2\) & \includegraphics[valign=m, scale=0.7]{bn.pdf} & \(\widetilde{C_n}\) & \includegraphics[valign=m, scale=0.7]{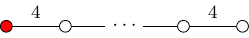}\\[1em]      
      \(D_n\) for \(n \geq 4\) & \includegraphics[valign=m, scale=0.7]{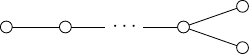} & \(\widetilde{D_n}\) & \includegraphics[valign=m, scale=0.7]{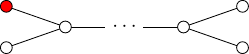}\\[1em]
      \(E_6\) & \includegraphics[valign=m, scale=0.7]{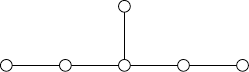} & \(\widetilde{E_6}\) & \includegraphics[valign=m, scale=0.7]{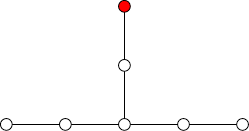}\\[2.5em]
      \(E_7\) & \includegraphics[valign=m, scale=0.7]{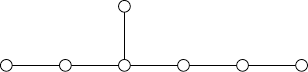} & \(\widetilde{E_7}\) & \includegraphics[valign=m, scale=0.7]{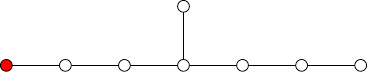}\\[1em]
      \(E_8\) & \includegraphics[valign=m, scale=0.7]{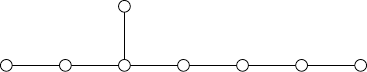} & \(\widetilde{E_8}\) & \includegraphics[valign=m, scale=0.7]{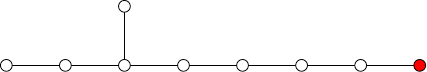}\\[1.5em]
      \(F_4\) & \includegraphics[valign=m, scale=0.7]{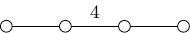} & \(\widetilde{F_4}\) & \includegraphics[valign=m, scale=0.7]{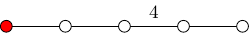}\\[1em]
      \(G_2\) & \includegraphics[valign=m, scale=0.7]{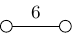} & \(\widetilde{G_2}\) & \includegraphics[valign=m, scale=0.7]{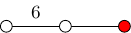}\\[1em]
      \(H_2\) & \includegraphics[valign=m, scale=0.7]{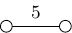}\\[1em]
      \(H_3\) & \includegraphics[valign=m, scale=0.7]{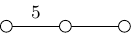}\\[1em]
      \(H_4\) & \includegraphics[valign=m, scale=0.7]{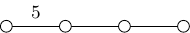}\\[1em]
      \(I_2(n)\) for \(n > 6\) & \includegraphics[valign=m, scale=0.7]{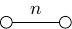}\\
      \bottomrule
    \end{tabular}
  \end{table}

  The first eight cases in~\cref{tab:finite-coxeter} are called \defn{crystallographic}.
  There is an important related notion: the \defn{affine Coxeter graphs} are obtained by attaching one extra ``affine node'' to each of the crystallographic Coxeter graphs.
  The position of the affine node is determined in each case.
  A Coxeter graph (or its associated Coxeter system) is called~\defn{affine} if \(\Gamma\) is one of the types listed in the second two columns of~\cref{tab:finite-coxeter}.
  The affine nodes are coloured red for emphasis.
  
\section{Regular languages and the Myhill--Nerode theorem}
In this section we recall the definition of regular languages, as well as a criterion for detecting a regular language that we will crucially use for the proof of~\cref{thm:all-coxeter-groups-main}.
\begin{definition}
  Let \(\Sigma\) be a finite alphabet.
  We say that a language \(L \subset \Sigma^{\ast}\) is \defn{regular} if there exists a deterministic finite state automaton (or DFA) \(M\) such that the language of words accepted by \(M\) is exactly \(L\).
\end{definition}
Recall again that a language is called regular if it is precisely the set of words accepted by some finite state automaton.
We do not review the definition and properties of finite state automata because we do not need anything further about them.
We refer the reader to e.g.~\cite[Chapter 1]{sip:06} for these details.

There is an explicit and elegant necessary and sufficient criterion to check whether a given language is regular.
Moreover, in case the language is regular, this criterion outputs a minimal deterministic finite automaton that recognises it.
The criterion is called the \emph{Myhill--Nerode theorem}, which we write below as~\Cref{thm:myhill-nerode}.

We first need the definition of a certain equivalence relation on words, which depends on a given language.
We call this the Myhill--Nerode equivalence relation.
\begin{definition}
  Let \(\Sigma\) be a finite alphabet, and let \(L \subset \Sigma^{\ast}\) be any language.
  For \(x, y \in \Sigma^{\ast}\), we say that \(x \sim_L y\) if for every string \(w \in \Sigma^{\ast}\), the string \(xw\) is in \(L\) if and only if the string \(yw\) is also in \(L\).
\end{definition}

We can now state the criterion, which is a fundamental result in the theory of regular languages.
See, e.g.~\cite[Problem 1.52]{sip:06}.
\begin{theorem}[Myhill--Nerode theorem]\label{thm:myhill-nerode}
  Let \(\Sigma^{\ast}\) be an alphabet, and let \(L \subset \Sigma^{\ast}\) be a language.
  Then \(L\) is regular if and only if the equivalence relation \(\sim_L\) has finitely many equivalence classes on \(\Sigma^{\ast}\).
  If \(L\) is regular, then there is a DFA \(M\) whose states are indexed by the equivalence classes of \(\sim_L\), which recognises \(M\).
  Furthermore, this is a DFA that recognises \(M\) that has the fewest possible states.
\end{theorem}

\section{Palindromic reduced words in certain Coxeter groups}
\label{sec:bad-subgraphs}
Recall from~\cref{sec:coxeter-basics} the classification of finite and affine Coxeter groups.
Recall that \(\pal(W)\) denotes the language of palindromic reduced words for a Coxeter system \((W,S)\).
We omit \(S\) from the notation for brevity, and when there is no ambiguity, we simply write \(\pal\).
In this section we show that \(\pal(W)\) is not regular for a particular set of Coxeter groups \(W\).
In the next section, we use the results of this section to prove the main theorem.

Before we prove irregularity, we prove a regularity result.
\begin{proposition}\label{prop:affine-a1}
  \(\pal\) is regular for the affine Coxeter system of type \(\widetilde{A_1}\).
\end{proposition}
\begin{proof}[Proof of~\cref{prop:affine-a1}]
  Let \(W\) be the Coxeter group of type \(\widetilde{A_1}\), which is the infinite dihedral group.
  Let \(s\) and \(t\) be the two Coxeter generators.
  In this case, an word \(\mathbf{w}\) is reduced if and only if the letters \(s\) and \(t\) alternate in it.
  Furthermore, \(\mathbf{w}\) is palindromic if and only if it begins and ends with the same letter.
  Thus a non-empty reduced word \(\mathbf{w}\) is palindromic if and only if it has an odd number of letters.
  If so desired, the empty word can also be regarded as a palindromic reduced word.
  The following deterministic finite automaton recognises \(\pal\) for the Coxeter group of type \(\widetilde{A_1}\).
  The states \(q_1\), \(q_s\), and \(q_t\) are the accepting states for the automaton.
  (If \(q_1\) is omitted from the list of accepting states, the automaton will only recognise the non-empty palindromic reduced words.)
  \begin{center}
    \begin{tikzpicture}
      \node[state, initial, accepting] (q0) {\(q_1\)};
      \node[state, accepting] (q1) [above right=1cm of q0] {\(q_s\)};
      \node[state] (q2) [right=2cm of q1] {\(q_{st}\)};      
      \node[state, accepting] (q3) [below right=1cm of q0] {\(q_{t}\)};
      \node[state] (q4) [right=2cm of q3] {\(q_{ts}\)};            
      \node[state] (q5) [below left=1cm of q2] {\(q_{\text{bad}}\)};
      \path[->, font=\small]
      (q0) edge node[left] {\(s\)} (q1)
      (q0) edge node[left] {\(t\)} (q3)
      (q1) edge[bend left=10] node[above] {\(t\)} (q2)
      (q1) edge node[left] {\(s\)} (q5)
      (q2) edge[bend left=10] node[below] {\(s\)} (q1)
      (q2) edge node[right] {\(t\)} (q5)
      (q3) edge[bend left=10] node[above]{\(s\)} (q4)
      (q3) edge node[left] {\(t\)} (q5)
      (q4) edge[bend left=10] node[below] {\(t\)} (q3)
      (q4) edge node[right] {\(s\)} (q5)
      (q5) edge [loop right] node {\(s,t\)} ();
      ;
    \end{tikzpicture}
  \end{center}
\end{proof}

The remainder of this section is devoted to proving irregularity results for all affine types with at least two vertices, as well as for the following three graphs.
\begin{equation}\label{eq:y3-z4-z5}
  \begin{split}
    Y_3 & = 
          \begin{tikzpicture}[scale=0.5]
            {
              \pgftransformxshift{2 cm}
              \draw (2 cm, 0 cm) edge node[above, font=\scriptsize]{\(\infty\)} +(2 cm,0);
              \draw (4.0 cm,0) -- +(2 cm,0);
              \draw[fill=white] (2 cm, 0 cm) circle (0.2cm);
              \draw[fill=white] (4 cm, 0 cm) circle (0.2cm);
              \draw[fill=white] (6 cm, 0 cm) circle (0.2cm);
            }
          \end{tikzpicture}\\
    Z_4 &=   \begin{tikzpicture}[scale=0.5]
      {
        \pgftransformxshift{2 cm}
        \draw (0 cm,0) -- (2 cm,0);
        \draw (2 cm, 0 cm) edge node[above, font=\scriptsize]{\(5\)} +(2 cm,0);
        \draw (4.0 cm,0) -- +(2 cm,0);
        \draw[fill=white] (0 cm, 0 cm) circle (0.2cm);
        \draw[fill=white] (2 cm, 0 cm) circle (0.2cm);
        \draw[fill=white] (4 cm, 0 cm) circle (0.2cm);
        \draw[fill=white] (6 cm, 0 cm) circle (0.2cm);
      }
    \end{tikzpicture}\\
    Z_5 &=   \begin{tikzpicture}[scale=0.5]
      \draw (0 cm,0) -- (2 cm,0);
      {
        \pgftransformxshift{2 cm}
        \draw (0 cm,0) -- (2 cm,0);
        \draw (2 cm, 0 cm) edge +(2 cm,0);
        \draw (4.0 cm,0) edge node [above, font=\scriptsize]{\(5\)} +(2 cm,0);
        \draw[fill=white] (0 cm, 0 cm) circle (0.2cm);
        \draw[fill=white] (2 cm, 0 cm) circle (0.2cm);
        \draw[fill=white] (4 cm, 0 cm) circle (0.2cm);
        \draw[fill=white] (6 cm, 0 cm) circle (0.2cm);
      }
      \draw[fill=white] (0 cm, 0 cm) circle (0.2cm);
    \end{tikzpicture}
  \end{split}
\end{equation}
While we prove the statements separately, all proofs follow the same schema.
The next lemma encapsulates the common proof schema.
Recall that \(|\mathbf{w}|\) is the number of letters in a word \(\mathbf{w}\).
Recall also that \(|\mathbf{w}|\) is greater than or equal to the length \(\ell(w)\) of the group element represented by \(\mathbf{w}\).
\begin{lemma}\label{lem:common-proof-schema}
  Let \((W,S)\) be a Coxeter system and let \(\pal\) be the language of its palindromic reduced words.
  Let \(\mathbf{u}\) and \(\mathbf{v}\) be words in \((W,S)\) representing group elements \(u\) and \(v\) respectively, and let \(s \in S\).
  Suppose that the following conditions hold.
  \begin{enumerate}[label=(C\arabic*)]
  \item\label{it:c-depth} For every \(r \in \mathbb{N}\), the depth of the root \(u^r v (\alpha_s)\) is exactly \(r \cdot |\mathbf{u}| + |\mathbf{v}| + 1\).
  \item\label{it:c-min-suffix} Let \(\mathbf{w}_r = \mathbf{u}^r \mathbf{v}\, s\). If \(\mathbf{w}_r'\) is any word such that \(\mathbf{w}_r \mathbf{w}_r' \in \pal\), then \(|\mathbf{w}_r'| > r\).
  \end{enumerate}
  Then the set \(\{\mathbf{w}_r \mid r \geq 0\}\) contains representatives of infinitely many Myhill--Nerode equivalence classes for \(L\).
\end{lemma}
\begin{proof}
  Suppose that we have words \(\mathbf{u},\mathbf{v}\) as well as \(s \in S\) satisfying the conditions~\ref{it:c-depth} and~\ref{it:c-min-suffix} of the lemma.
  Let \(\mathbf{w}_r = \mathbf{u}^r\mathbf{v} s\).
  Let \(a = |\mathbf{u}|\) and \(b = |\mathbf{v}|\).
  By~\ref{it:c-depth} and~\cref{prop:depth-of-reflection}, we note that for any \(r \in \mathbb{N}\), the length of the element \(w_r v^{-1}u^{-r}\) is exactly \(2ra + 2b + 1\).
  Since this is exactly equal to the number of letters in the word \(\mathbf{u}^r\mathbf{v} s \overline{\mathbf{v}}\,\overline{\mathbf{u}}^{r}\), this is a reduced word.
  It is also clearly palindromic, and thus \(\mathbf{w}_r\overline{\mathbf{v}}\,\overline{\mathbf{u}}^{r} \in L\).
  In particular, \(\mathbf{w}_r' = \overline{\mathbf{v}}\,\overline{\mathbf{u}}^r\) is a word with \(ra + b\) letters such that \(\mathbf{w}_r\mathbf{w}_r' \in \pal\).

  Now fix some \(c \in \mathbb{N}\) such that \(ra + b \leq rc\) for each \(r \geq 1\).
  Choose any infinite subsquence \((\mathbf{w}_{r_j})\) of \((\mathbf{w}_j)_{j \ge 1}\) such that \(r_j > cr_i\) for \(i < j\).
  Consider any pair \(i < j\).
  Then as before, there is a word \(\mathbf{w}_{r_i}'\) with \(r_ia + b\) letters such that \(\mathbf{w}_{r_i}\mathbf{w}_{r_i}' \in L\).
  However, \(r_ia + b < r_ic < r_j\) by design.
  Thus by~\ref{it:c-min-suffix}, we know that \(\mathbf{w}_{r_j}\mathbf{w}_{r_i}' \notin L\).
  Thus \(\mathbf{w}_{r_i} \not \sim_L \mathbf{w}_{r_j}\).
  This argument shows that the words \((\mathbf{w}_{r_i})\) are pairwise inequivalent under the Myhill--Nerode equivalence relation for \(L\).
\end{proof}

  \begin{table}[h]
  \centering
  \caption{Imaginary roots \(\delta\) as well as words \(\mathbf{u}\) and simple reflections \(s\) for each affine type other than \(\widetilde{A_1}\) satisfying the conditions of~\cref{lem:common-proof-schema}.
  In \(\mathbf{u}\), the tuple \((i_1,\ldots,i_r)\) denotes the word \(s_{i_1}\cdots s_{i_r}\).}\label{tab:units}
  \begin{tabular}{lccc}
    Type & Labelled Coxeter graph with \(\delta\) inset & \(\mathbf{u}\) & \(s\)\\
    \toprule
    \(\widetilde{A_n}\):
         & \includegraphics[valign=m]{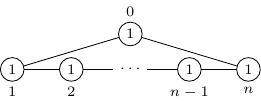}
                                             & \((0,n,n-1,\ldots, 2,1)^n\) & \(s_0\)\\ 
    \midrule

    \(\widetilde{B_n}\):
         & \includegraphics[valign=m]{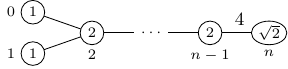}
                                             & \(\makecell{(2,3, \ldots ,{n-1},n,\\ {n-1},\ldots,3,1,0)}\) & \(s_2\)\\
    \midrule

    \(\widetilde{C_n}\):
           & \includegraphics[valign=m]{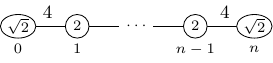}
         &\(\makecell{(1, 2, \ldots, {n-1}, n,\\ {n-1}, \ldots ,3,2,0)}\) & \(s_1\)
    \\
    \midrule

    \(\widetilde{D_n}\):
           &\includegraphics[valign=m]{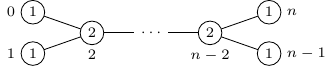}
    &\(\makecell{(2,3, \ldots ,{n-2},n,{n-1},\\{n-2} \ldots ,3,1,0)}\)& \(s_2\)

    \\
    \midrule

    \(\widetilde{E_6}\):
           &\includegraphics[valign=m]{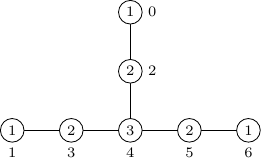}           
                                   &\(\makecell{(4,5,6,3,1,2,\\0,4,5,3,2)}\) & \(s_4\)

    \\
    \midrule

    \(\widetilde{E_7}\):
           &\includegraphics[valign=m]{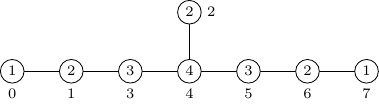}
                                   &\(\makecell{(4,5,6,7,3,1,0,2,\\4,5,6,3,1,4,2,5,3)}\) & \(s_4\)\\
    \midrule
    \(\widetilde{E_8}\):
           &\includegraphics[valign=m]{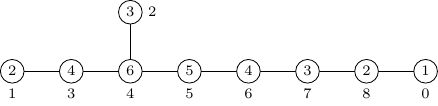}           
                                   & \(\makecell{(4,5,6,7,8,0,3,1,2,4,5,\\6,7,8,3,4,5,6,7,2,4,\\3,1,5,6,4,3,2,5)}\)& \(s_4\)\\
    \midrule
    \(\widetilde{F_4}\):
&\includegraphics[valign=m]{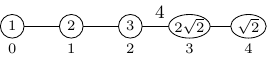}                           
                                   &\((2,3,4,1,0,2,1,3)\)&\(s_2\)\\

    \midrule
    \(\widetilde{G_2}\):&\includegraphics[valign=m]{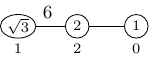}
    &\((2,1,0)\)&\(s_2\)\\
    \bottomrule

  \end{tabular}
\end{table}

Now we tackle the affine cases with at least two vertices, which are exactly the affine cases other than \(\widetilde{A_1}\).
\begin{proposition}\label{prop:affine-types}
  For any affine Coxeter system other than type \(\widetilde{A_1}\), the language \(\pal\) is not regular.
\end{proposition}
\begin{proof}[Proof of~\cref{prop:affine-types}]
  In~\cref{tab:units} we list the affine Coxeter systems with at least two vertices.
  In each case, we have labelled the vertices of the Coxeter graph from \(0\) to \(n\).
  We also draw inside each node a coefficient, such that the sum corresponds to a \defn{positive imaginary root} \(\delta\).
  For our purposes, \(\delta\) is a non-zero vector with positive coefficients in the basis of simple roots of \(V_W\) that is stable under the \(W\)-action.
  It is well-known that in affine type, the subspace of \(V_W\) that is pointwise preserved by the \(W\)-action is one-dimensional (this follows from, e.g.~\cite[Section 2.6]{hum:90}).
  Thus \(\delta\) is unique upto scaling.
  
  For each Coxeter graph, we also list a word \(\mathbf{u}\) and an \(s \in S\).
  For clarity, we write \(i\) in place of \(s_i\) in the word \(\mathbf{u}\).
  Finally, in case, we set \(\mathbf{v} = 1\).
  We claim that in each case, the words \(\mathbf{u}\), \(\mathbf{v} = 1\), together with \(s\) satisfy the two conditions of~\cref{lem:common-proof-schema}, thus proving the result.

  The first condition~\ref{it:c-depth} says that the depth of \(u^r v (\alpha_s)\) is exactly \(r \cdot |\mathbf{u}| + |\mathbf{v}| + 1\), which is equal to \(r \cdot |\mathbf{u}| + 1\) in our case.
  We use the dual numbers game and~\cref{cor:numbers-game-depth} to check that the depth increases by one after applying each letter of \(\mathbf{u}^r\) to \(\alpha_s\), which will prove~\ref{it:c-depth}.
  
  We will see in each case that \(u(\alpha) = \alpha + c \delta\) for some constant \(c\).
  Since \(s_i(\delta) = \delta\) for each \(\delta\), we have \(s_i(u(\alpha)) = s_i(\alpha) + c\delta\).
  Thus if the depth increases at each step while applying \(\mathbf{u}\) to \(\alpha\), the same continues to hold for subsequent applications of \(\mathbf{u}\).
  Thus we only need to check the depth increase for \(u(\alpha)\).

  The calculations are not difficult, but we sketch a selection for the convenience of the reader.
  For instance, for type \(\widetilde{A_n}\), we start with \(\alpha_0\) and apply the word \(\mathbf{u}\) in blocks as below.
  \begin{align*}
    \begin{tikzpicture}[scale=0.5, baseline=(current bounding box.center), font=\scriptsize]
      \draw (0 cm,0) -- (8 cm,0);
      \draw (0 cm,0) -- (4.0 cm, 1.2 cm);
      \draw (4.0 cm, 1.2 cm) -- (8 cm, 0);
      \draw[fill=white] (0 cm, 0 cm) circle (0.4cm);
      \draw[fill=white] (2 cm, 0 cm) circle (0.4cm);
      \node[rectangle, draw=none, fill=white] at (4cm,0cm) {\(\cdots\)};
      \draw[fill=white] (6 cm, 0 cm) circle (0.4cm);
      \draw[fill=white] (8 cm, 0 cm) circle (0.4cm);
      \draw[fill=white] (4.0 cm, 1.2 cm) circle (0.4cm) node {\(1\)};
    \end{tikzpicture}
    &\xmapsto{(n,n-1,\ldots, 2,1)}
    \begin{tikzpicture}[scale=0.5, baseline=(current bounding box.center), font=\scriptsize]
      \draw (0 cm,0) -- (8 cm,0);
      \draw (0 cm,0) -- (4.0 cm, 1.2 cm);
      \draw (4.0 cm, 1.2 cm) -- (8 cm, 0);
      \draw[fill=white] (0 cm, 0 cm) circle (0.4cm) node {\(1\)};
      \draw[fill=white] (2 cm, 0 cm) circle (0.4cm) node {\(1\)};
      \node[rectangle, draw=none, fill=white] at (4cm,0cm) {\(\cdots\)};
      \draw[fill=white] (6 cm, 0 cm) circle (0.4cm) node {\(1\)};
      \draw[fill=white] (8 cm, 0 cm) circle (0.4cm) node {\(2\)};
      \draw[fill=white] (4.0 cm, 1.2 cm) circle (0.4cm) node {\(1\)};
    \end{tikzpicture}\\
    &\xmapsto{(n-1,\ldots, 2,1,0)}
    \begin{tikzpicture}[scale=0.5, baseline=(current bounding box.center), font=\scriptsize]
      \draw (0 cm,0) -- (8 cm,0);
      \draw (0 cm,0) -- (4.0 cm, 1.2 cm);
      \draw (4.0 cm, 1.2 cm) -- (8 cm, 0);
      \draw[fill=white] (0 cm, 0 cm) circle (0.4cm) node {\(2\)};
      \draw[fill=white] (2 cm, 0 cm) circle (0.4cm) node {\(2\)};
      \node[rectangle, draw=none, fill=white] at (4cm,0cm) {\(\cdots\)};
      \draw[fill=white] (6 cm, 0 cm) circle (0.4cm) node {\(3\)};
      \draw[fill=white] (8 cm, 0 cm) circle (0.4cm) node {\(2\)};
      \draw[fill=white] (4.0 cm, 1.2 cm) circle (0.4cm) node {\(2\)};
    \end{tikzpicture}\\
    &\quad \vdots \\
    &\xmapsto{(0,n,n-1,\ldots, 2)}
    \begin{tikzpicture}[scale=0.5, baseline=(current bounding box.center), font=\scriptsize]
      \draw (0 cm,0) -- (8 cm,0);
      \draw (0 cm,0) -- (4.0 cm, 1.2 cm);
      \draw (4.0 cm, 1.2 cm) -- (8 cm, 0);
      \draw[fill=white] (0 cm, 0 cm) circle (0.4cm) node {\(n\)};
      \draw[fill=white] (2 cm, 0 cm) circle (0.4cm) node {\(n\)};
      \node[rectangle, draw=none, fill=white] at (4cm,0cm) {\(\cdots\)};
      \draw[fill=white] (6 cm, 0 cm) circle (0.4cm) node {\(n\)};
      \draw[fill=white] (8 cm, 0 cm) circle (0.4cm) node {\(n\)};
      \draw[fill=white] (4.0 cm, 1.2 cm) ellipse (0.6cm) node {\(n+1\)};
    \end{tikzpicture}
  \end{align*}
  Thus \(u(\alpha_0) = \alpha_0 + n \delta\) as desired.
  We briefly sketch similar calculations for the other infinite families \(BCD\).
  The following is \(\widetilde{B_n}\).
  \begin{align*}
\begin{tikzpicture}[scale=0.5, baseline=(current bounding box.center), font=\scriptsize]
  \draw (0,0.7 cm) -- (2 cm,0);
  \draw (0,-0.7 cm) -- (2 cm,0);
  \draw (2 cm,0) -- (6 cm,0);
  \draw (6 cm, 0 cm) edge node[above, font=\scriptsize]{\(4\)} +(2 cm,0);
  \draw[fill=white] (0 cm, 0.7 cm) circle (0.4cm);
  \draw[fill=white] (0 cm, -0.7 cm) circle (0.4cm);
  \draw[fill=white] (2 cm, 0 cm) circle (0.4cm) node {\(1\)};
  \node[rectangle, draw=none, fill=white] at (4cm,0cm) {\(\cdots\)};    
  \draw[fill=white] (6 cm, 0 cm) circle (0.4cm);
  \draw[fill=white] (8 cm, 0 cm) circle (0.4cm);
\end{tikzpicture}
    &\xmapsto{(n,n-1,\ldots,3)}
\begin{tikzpicture}[scale=0.5, baseline=(current bounding box.center), font=\scriptsize]
  \draw (0,0.7 cm) -- (2 cm,0);
  \draw (0,-0.7 cm) -- (2 cm,0);
  \draw (2 cm,0) -- (6 cm,0);
  \draw (6 cm, 0 cm) edge node[above, font=\scriptsize]{\(4\)} +(2 cm,0);
  \draw[fill=white] (0 cm, 0.7 cm) circle (0.4cm);
  \draw[fill=white] (0 cm, -0.7 cm) circle (0.4cm);
  \draw[fill=white] (2 cm, 0 cm) circle (0.4cm) node {\(1\)};
  \node[rectangle, draw=none, fill=white] at (4cm,0cm) {\(\cdots\)};    
  \draw[fill=white] (6 cm, 0 cm) circle (0.4cm) node {\(1\)};
  \draw[fill=white] (8 cm, 0 cm) circle (0.4cm) node {\(1\)};
\end{tikzpicture}\\
    & \xmapsto{(2,3,\ldots,n-1)}
\begin{tikzpicture}[scale=0.5, baseline=(current bounding box.center), font=\scriptsize]
  \draw (0,0.7 cm) -- (2 cm,0);
  \draw (0,-0.7 cm) -- (2 cm,0);
  \draw (2 cm,0) -- (6 cm,0);
  \draw (6 cm, 0 cm) edge node[above, font=\scriptsize]{\(4\)} +(2 cm,0);
  \draw[fill=white] (0 cm, 0.7 cm) circle (0.4cm) node {\(1\)};
  \draw[fill=white] (0 cm, -0.7 cm) circle (0.4cm) node {\(1\)};
  \draw[fill=white] (2 cm, 0 cm) circle (0.4cm) node {\(3\)};
  \node[rectangle, draw=none, fill=white] at (4cm,0cm) {\(\cdots\)};    
  \draw[fill=white] (6 cm, 0 cm) circle (0.4cm) node {\(2\)};
  \draw[fill=white] (8 cm, 0 cm) circle (0.4cm) node {\(2\)};
\end{tikzpicture}      
  \end{align*}
  In this case, \(u(\alpha_2) = \alpha_2 + \delta\).

  The following is \(\widetilde{C_n}\).
  \begin{align*}
    \begin{tikzpicture}[scale=0.5, baseline=(current bounding box.center), font=\scriptsize]
      \draw (0, 0 cm) edge node[above, font=\scriptsize]{\(4\)} +(2 cm,0);
      {
        \pgftransformxshift{2 cm}
        \draw (0 cm,0) -- (4 cm,0);
        \draw (4 cm, 0 cm) edge node[above, font=\scriptsize]{\(4\)} +(2 cm,0);
        \draw[fill=white] (0 cm, 0 cm) circle (0.4cm) node {\(1\)};
        \node[rectangle, draw=none, fill=white] at (2cm,0cm) {\(\cdots\)};    
        \draw[fill=white] (4 cm, 0 cm) circle (0.4cm);
        \draw[fill=white] (6 cm, 0 cm) circle (0.4cm);
      }
      \draw[fill=white] (0 cm, 0 cm) circle (0.4cm);
    \end{tikzpicture}
    &\xmapsto{(n,n-1,\ldots,3,2,0)}
    \begin{tikzpicture}[scale=0.5, baseline=(current bounding box.center), font=\scriptsize]
      \draw (0, 0 cm) edge node[above, font=\scriptsize]{\(4\)} +(2 cm,0);
      {
        \pgftransformxshift{2 cm}
        \draw (0 cm,0) -- (4 cm,0);
        \draw (4 cm, 0 cm) edge node[above, font=\scriptsize]{\(4\)} +(2 cm,0);
        \draw[fill=white] (0 cm, 0 cm) circle (0.4cm) node {\(1\)};
        \node[rectangle, draw=none, fill=white] at (2cm,0cm) {\(\cdots\)};    
        \draw[fill=white] (4 cm, 0 cm) circle (0.4cm) node {\(1\)};
        \draw[fill=white] (6 cm, 0 cm) ellipse (0.6cm and 0.4cm) node {\(\sqrt{2}\)};
      }
      \draw[fill=white] (0 cm, 0 cm) ellipse (0.6cm and 0.4cm) node {\(\sqrt{2}\)};
    \end{tikzpicture}\\
    &\xmapsto{(1,2,\ldots,n-1)}
    \begin{tikzpicture}[scale=0.5, baseline=(current bounding box.center), font=\scriptsize]
      \draw (0, 0 cm) edge node[above, font=\scriptsize]{\(4\)} +(2 cm,0);
      {
        \pgftransformxshift{2 cm}
        \draw (0 cm,0) -- (4 cm,0);
        \draw (4 cm, 0 cm) edge node[above, font=\scriptsize]{\(4\)} +(2 cm,0);
        \draw[fill=white] (0 cm, 0 cm) circle (0.4cm) node {\(3\)};
        \node[rectangle, draw=none, fill=white] at (2cm,0cm) {\(\cdots\)};    
        \draw[fill=white] (4 cm, 0 cm) circle (0.4cm) node {\(2\)};
        \draw[fill=white] (6 cm, 0 cm) ellipse (0.6cm and 0.4cm) node {\(\sqrt{2}\)};
      }
      \draw[fill=white] (0 cm, 0 cm) ellipse (0.6cm and 0.4cm) node {\(\sqrt{2}\)};
    \end{tikzpicture}
  \end{align*}
  In this case, \(u(\alpha_1) = \alpha_1 + \delta\).

  The following is \(\widetilde{D_n}\).
  \begin{align*}
    \begin{tikzpicture}[scale=0.5,baseline=(current bounding box.center), font=\scriptsize]
      \draw (0,0.7 cm) -- (2 cm,0);
      \draw (0,-0.7 cm) -- (2 cm,0);
      \draw (2 cm,0) -- (6 cm,0);
      \draw (6 cm,0) -- (8 cm,0.7 cm);
      \draw (6 cm,0) -- (8 cm,-0.7 cm);
      \draw[fill=white] (0 cm, 0.7 cm) circle (0.4cm) node {\(1\)};
      \draw[fill=white] (0 cm, -0.7 cm) circle (0.4cm);
      \draw[fill=white] (2 cm, 0 cm) circle (0.4cm);
      \node[rectangle, draw=none, fill=white] at (4cm,0cm) {\(\cdots\)};    
      \draw[fill=white] (6 cm, 0 cm) circle (0.4cm);
      \draw[fill=white] (8 cm, 0.7 cm) circle (0.4cm);
      \draw[fill=white] (8 cm, -0.7 cm) circle (0.4cm);
    \end{tikzpicture}
    &\xmapsto{(n,n-1,n-2, \ldots, 3,1,1)}
    \begin{tikzpicture}[scale=0.5,baseline=(current bounding box.center), font=\scriptsize]
      \draw (0,0.7 cm) -- (2 cm,0);
      \draw (0,-0.7 cm) -- (2 cm,0);
      \draw (2 cm,0) -- (6 cm,0);
      \draw (6 cm,0) -- (8 cm,0.7 cm);
      \draw (6 cm,0) -- (8 cm,-0.7 cm);
      \draw[fill=white] (0 cm, 0.7 cm) circle (0.4cm) node {\(1\)};
      \draw[fill=white] (0 cm, -0.7 cm) circle (0.4cm) node {\(1\)};
      \draw[fill=white] (2 cm, 0 cm) circle (0.4cm) node {\(1\)};
      \node[rectangle, draw=none, fill=white] at (4cm,0cm) {\(\cdots\)};    
      \draw[fill=white] (6 cm, 0 cm) circle (0.4cm) node {\(1\)};
      \draw[fill=white] (8 cm, 0.7 cm) circle (0.4cm) node {\(1\)};
      \draw[fill=white] (8 cm, -0.7 cm) circle (0.4cm) node {\(1\)};
    \end{tikzpicture}\\
    &\xmapsto{(2,3,\ldots,n-2)}
    \begin{tikzpicture}[scale=0.5, baseline=(current bounding box.center), font=\scriptsize]
      \draw (0,0.7 cm) -- (2 cm,0);
      \draw (0,-0.7 cm) -- (2 cm,0);
      \draw (2 cm,0) -- (6 cm,0);
      \draw (6 cm,0) -- (8 cm,0.7 cm);
      \draw (6 cm,0) -- (8 cm,-0.7 cm);
      \draw[fill=white] (0 cm, 0.7 cm) circle (0.4cm) node {\(1\)};
      \draw[fill=white] (0 cm, -0.7 cm) circle (0.4cm) node {\(1\)};
      \draw[fill=white] (2 cm, 0 cm) circle (0.4cm) node {\(3\)};
      \node[rectangle, draw=none, fill=white] at (4cm,0cm) {\(\cdots\)};    
      \draw[fill=white] (6 cm, 0 cm) circle (0.4cm) node {\(2\)};
      \draw[fill=white] (8 cm, 0.7 cm) circle (0.4cm) node {\(1\)};
      \draw[fill=white] (8 cm, -0.7 cm) circle (0.4cm) node {\(1\)};
    \end{tikzpicture}      
  \end{align*}
  In this case, \(u(\alpha_2) = \alpha_2 + \delta\).

  We leave the remaining (finite) calculations to the reader, observing in each case that we obtain \(u(\alpha_s) = \alpha_s + \delta\) of depth exactly equal to \(|\mathbf{u}| + 1\) as desired.
  Thus the first condition~\ref{it:c-depth} of~\cref{lem:common-proof-schema} is proved.

  Let us now prove~\ref{it:c-min-suffix} of~\cref{lem:common-proof-schema}.
  Let \(\mathbf{w}_r = \mathbf{u}^rs\) in each case.
  Suppose that there is some word \(\mathbf{w}_r'\) with \(|\mathbf{w}_r'| < r\), such that \(\mathbf{w}_r\mathbf{w}_r' \in \pal\).
  Then the word \(\mathbf{w}_r\) begins with the reverse of \(\mathbf{w}_r'\).
  We can then ``cancel'' the reverse of \(\mathbf{w}_r'\) and \(\mathbf{w}_r'\) from the front and back of \(\mathbf{w}_r\mathbf{w}_r'\) respectively.
  This leaves behind another palindromic reduced word, which necessarily has the form
  \(\mathbf{x} \mathbf{u}^m s\) for some \(m \geq 0\), where \(\mathbf{x}\) is a tail of \(\mathbf{u}\) with less than \(|\mathbf{u}|\) letters.
  In each case listed, it is evident by inspection that no such word can be palindromic.
  Thus the second condition~\ref{it:c-min-suffix} of~\cref{lem:common-proof-schema} is also proved.

  We conclude that \(L\) is not a regular language for any affine Coxeter graph with at least two vertices.
\end{proof}

Next we tackle the case of \(Y_3\) drawn in~\eqref{eq:y3-z4-z5}.
\begin{proposition}\label{prop:y3}
  The language of palindromic reduced words for the Coxeter graph \(Y_3\) is not regular.
\end{proposition}
\begin{proof}
  Let \(\pal\) be the language of palindromic reduced words for \(Y_3\).
  Label the nodes of \(Y_3\) as follows.
\begin{center}
  \begin{tikzpicture}[scale=0.5]
    {
      \pgftransformxshift{2 cm}
      \draw (2 cm, 0 cm) edge node[above, font=\scriptsize]{\(\infty\)} +(2 cm,0);
      \draw (4.0 cm,0) -- +(2 cm,0);
      \draw[fill=white] (2 cm, 0 cm) circle (0.2cm) node[below=4pt] {1};
      \draw[fill=white] (4 cm, 0 cm) circle (0.2cm) node[below=4pt] {2};
      \draw[fill=white] (6 cm, 0 cm) circle (0.2cm) node[below=4pt] {3};
    }
  \end{tikzpicture}
\end{center}
Set \(\mathbf{u} = s_1s_2\) and \(\mathbf{v} = 1\). Let \(s = s_3\).
We will check that \(\mathbf{u}, \mathbf{v}, s\) as above satisfy the two conditions~\ref{it:c-depth} and~\ref{it:c-min-suffix} of~\cref{lem:common-proof-schema}, thereby concluding the result of the proposition.

To check~\ref{it:c-depth}, we need to show that the depth of \(u^r(\alpha_s)\) is exactly \(2r + 1\).
Write an arbitrary element \(a \alpha_1 + b \alpha_2 + c \alpha_3\) of the root space as the vector \((a,b,c)\).
Then
\[
  s_2((a,b,c)) = (a, 2a+c - b, c) \text{ and } s_1((a,b,c)) = (2b - a, b , c).
\]
Thus we can show by an easy induction argument that \[u^r(\alpha_3) = u^r((0,0,1)) = (r(r+1), r^2, 1) \text{ and } s_2u^r(\alpha_3) = (r(r+1), (r+1)^2, 1).\]
In particular, the depth of \(u^r(\alpha_3)\) is precisely \(2r+1\) by~\cref{cor:numbers-game-depth}.
This argument proves~\ref{it:c-depth}.

To prove~\ref{it:c-min-suffix}, suppose for contradiction that there is some word \(\mathbf{w}_r'\) with \(|\mathbf{w}_r'| < r\), such that \(\mathbf{w}_r\mathbf{w}_r' \in \pal\).
Then we can ``cancel'' \(\mathbf{w}_r'\) as well as the reverse of \(\mathbf{w}_r'\) from the back and front of \(\mathbf{w}_r\mathbf{w}_r'\), leaving behind a palindromic reduced word.
This new word has the form \(\mathbf{x}\mathbf{u}^ms_3\), where \(\mathbf{x}\) is either empty or equal to \(s_2\).
No such word can be palindromic, which is a contradiction.
Thus~\ref{it:c-min-suffix} is proved as well.
\end{proof}
Next, we handle \(Z_4\), drawn in~\eqref{eq:y3-z4-z5}.
\begin{proposition}\label{prop:z4}
  The language of palindromic regular words for the Coxeter graph \(Z_4\) is not regular.
\end{proposition}
\begin{proof}
  Let \(\pal\) be the language of palindromic reduced words for \(Z_4\).
  Label the nodes of \(Z_4\) as follows.
  \begin{center}
    \begin{tikzpicture}[scale=0.5]
      {
        \pgftransformxshift{2 cm}
        \draw (0 cm,0) -- (2 cm,0);
        \draw (2 cm, 0 cm) edge node[above, font=\scriptsize]{\(5\)} +(2 cm,0);
        \draw (4.0 cm,0) -- +(2 cm,0);
        \draw[fill=white] (0 cm, 0 cm) circle (0.2cm) node[below=4pt] {\(1\)};
        \draw[fill=white] (2 cm, 0 cm) circle (0.2cm) node[below=4pt] {\(2\)};
        \draw[fill=white] (4 cm, 0 cm) circle (0.2cm) node[below=4pt] {\(3\)};
        \draw[fill=white] (6 cm, 0 cm) circle (0.2cm) node[below=4pt] {\(4\)};
      }
    \end{tikzpicture}
  \end{center}
  Let \(\tau = (1 + \sqrt{5})/2\) be the golden ratio.
  Recall that \(\tau\) satisfies the equation \(\tau^2 = \tau + 1\).
  Let \[\mathbf{u} = s_1s_4s_2s_3s_2s_3s_2, \quad \mathbf{v} = s_4s_1s_2s_3, \text{ and }s = s_2.\]
  For \(r \in \mathbb{N}\), set \(\mathbf{w}_r = \mathbf{u}^r\mathbf{v}s_2\).
  We show that \(\mathbf{u}, \mathbf{v}, s\) as above satisfy the two conditions~\ref{it:c-depth} and~\ref{it:c-min-suffix} of~\cref{lem:common-proof-schema}, thereby proving the proposition.
  
  To prove~\ref{it:c-depth}, we need to show that the depth of \(u^r v (\alpha_2)\) is \(7r + 5\).
  We use the dual numbers game again.
  First observe the following.
  \[\begin{tikzpicture}[scale=0.5, font=\scriptsize]
      {
        \pgftransformxshift{2 cm}
        \draw (0 cm,0) -- (2 cm,0);
        \draw (2 cm, 0 cm) edge node[above, font=\scriptsize]{\(5\)} +(2 cm,0);
        \draw (4.0 cm,0) -- +(2 cm,0);
        \draw[fill=white] (0 cm, 0 cm) circle (0.4cm) node {\(0\)};
        \draw[fill=white] (2 cm, 0 cm) circle (0.4cm) node {\(1\)};
        \draw[fill=white] (4 cm, 0 cm) circle (0.4cm) node {\(0\)};
        \draw[fill=white] (6 cm, 0 cm) circle (0.4cm) node {\(0\)};
      }
    \end{tikzpicture}
    \quad
    \xmapsto{v}
    \quad
    \begin{tikzpicture}[scale=0.5, font=\scriptsize]
      {
        \pgftransformxshift{2 cm}
        \draw (0 cm,0) -- (2 cm,0);
        \draw (2 cm, 0 cm) edge node[above, font=\scriptsize]{\(5\)} +(2 cm,0);
        \draw (4.0 cm,0) -- +(2 cm,0);
        \draw[fill=white] (0 cm, 0 cm) circle (0.4cm) node {\(\tau\)};
        \draw[fill=white] (2 cm, 0 cm) circle (0.4cm) node {\(\tau\)};
        \draw[fill=white] (4 cm, 0 cm) circle (0.4cm) node {\(\tau\)};
        \draw[fill=white] (6 cm, 0 cm) circle (0.4cm) node {\(\tau\)};
      }
    \end{tikzpicture}
  \]
  It is easy to check that \(u\) preserves the subspace of the root space consisting of vectors of the form \((a,b,b,a)\).
  Indeed, we examine the result of applying each letter of \(\mathbf{u}\) to a vector of this form.
  \begin{center}
    \begin{tabular}{l|llll}
      Letter & \(1\) & \(2\) & \(3\) & \(4\)\\
      \hline
             & \(a\) & \(b\) & \(b\) & \(a\)\\
      \(s_2\) & \(a\) & \textcolor{darkred}{\(a + b(\tau-1)\)} & \(b\) & \(a\)\\
      \(s_3\) & \(a\) & \(a + b(\tau-1)\) & \textcolor{darkred}{\(a(\tau + 1)\)} & \(a\)\\
      \(s_2\) & \(a\) & \textcolor{darkred}{\(a(2\tau+1) + b(1-\tau)\)} & \(a(\tau + 1)\) & \(a\)\\
      \(s_3\) & \(a\) & \(a(2\tau+1) + b(1-\tau)\) & \textcolor{darkred}{\(a(2\tau + 2) -b\)} & \(a\)\\
      \(s_2\) & \(a\) & \textcolor{darkred}{\(a(2\tau + 2) - b\)} & \(a(2\tau + 2) -b\) & \(a\)\\
      \(s_4\) & \(a\) & \(a(2\tau + 2) - b\) & \(a(2\tau + 2) -b\) & \textcolor{darkred}{\(a(2\tau + 1) -b\)}\\
      \(s_1\) & \textcolor{darkred}{\(a(2\tau + 1) -b\)} & \(a(2\tau + 2) - b\) & \(a(2\tau + 2) -b\) & \(a(2\tau + 1) -b\)\\
    \end{tabular}
  \end{center}
  The depth strictly increases after applying each letter of \(\mathbf{u}\) if and only if the coordinate that changes at each step is strictly greater than the original value.
  We thus obtain the following inequalities.
  \begin{align}
    \label{eq:cone-inequalities-z4}
    \begin{split}
      (a + b(\tau-1)) - b &> 0\\
      a(\tau + 1) - b &> 0\\
      (a(2\tau + 1) + b(1-\tau)) - (a + b (\tau-1)) &> 0\\
      (a(2\tau + 2) - b) - a(\tau+1) &> 0\\
      (a(2\tau + 2) - b) - (a(2\tau + 1) + b (1-\tau)) & > 0\\
      (a(2\tau + 1) - b) - a & > 0.
    \end{split}
  \end{align}
  To check whether the depth increases at each step, we obtain the following inqualities.

  Let \(C\) be the cone cut out by the inequalities~\eqref{eq:cone-inequalities-z4}.
  It is not hard to check that the initial vector \((\tau, \tau, \tau, \tau)\) lies in \(C\).
  Moreover, it is also easy to check that \(C\) is stable under the application of \(u\).
  Thus the depth of \(u^rv(\alpha_2) = 7r+5\), and~\ref{it:c-depth} is proved.

  To prove~\ref{it:c-min-suffix},
  suppose for contradiction that there is some word \(\mathbf{w}_r\) with \(|\mathbf{w}_r'| < r\), such that \(\mathbf{w}_r\mathbf{w}_r' \in \pal\).
  Then we can ``cancel'' \(\mathbf{w}_r'\) with the reverse of \(\mathbf{w}_r'\) from the front and back of \(\mathbf{w}_r\mathbf{w}_r'\) to yield the palindromic word
  \(\mathbf{x} \mathbf{u}^m \mathbf{v} s_2\)
  for some \(m < r\), where \(\mathbf{x}\) is a tail of \(\mathbf{u}\) of length at most \(6\).
  However, it is clear from the forms of \(\mathbf{u}\) and \(\mathbf{v}\) that no such word can be palindromic, and~\ref{it:c-min-suffix} is proved.
\end{proof}
Finally, we handle \(Z_5\), drawn in~\eqref{eq:y3-z4-z5}.
\begin{proposition}\label{prop:z5}
  The language of palindromic reduced expressions for the Coxeter graph \(Z_5\) is not regular. 
\end{proposition}
\begin{proof}
  Let \(\pal\) be the language of palindromic regular expressions for \(Z_5\).
  Label the nodes of \(Z_5\) as follows.
  \begin{center}
    \begin{tikzpicture}[scale=0.5]
      \draw (0 cm,0) -- (2 cm,0);
      {
        \pgftransformxshift{2 cm}
        \draw (0 cm,0) -- (2 cm,0);
        \draw (2 cm, 0 cm) edge +(2 cm,0);
        \draw (4.0 cm,0) edge node [above, font=\scriptsize]{\(5\)} +(2 cm,0);
        \draw[fill=white] (0 cm, 0 cm) circle (0.2cm) node[below=4pt] {\(2\)};
        \draw[fill=white] (2 cm, 0 cm) circle (0.2cm) node[below=4pt] {\(3\)};
        \draw[fill=white] (4 cm, 0 cm) circle (0.2cm) node[below=4pt] {\(4\)};
        \draw[fill=white] (6 cm, 0 cm) circle (0.2cm)  node[below=4pt] {\(5\)};
      }
      \draw[fill=white] (0 cm, 0 cm) circle (0.2cm)  node[below=4pt] {\(1\)};
    \end{tikzpicture}
  \end{center}
  Recall that \(\tau = (1+ \sqrt{5})/2\) is the golden ratio.
  Let \[\mathbf{u} = s_5s_4s_5s_4s_1s_2s_3, \quad \mathbf{v} = s_4s_5s_4s_1s_2s_3s_5s_4s_5s_4s_1s_2, \text{ and }s = s_3.\]
  For \(r \in \mathbb{N}\), set \(\mathbf{w}_r = \mathbf{u}^r\mathbf{v}s_2\).
  We show that \(\mathbf{u}, \mathbf{v}, s\) as above satisfy the two conditions~\ref{it:c-depth} and~\ref{it:c-min-suffix} of~\cref{lem:common-proof-schema}, thereby proving the proposition.  

  We prove~\ref{it:c-depth} by checking that the depth of \(u^r v (\alpha_3)\) is exactly \(7r + 13\).
  Once again, we use the dual numbers game of~\cite[Section 4.6]{bjo.bre:05}.
  First observe the following by direct calculation.
  \begin{equation}\label{eq:v-alpha-3}
    \begin{tikzpicture}[scale=0.6, font=\scriptsize]
      \draw (0 cm,0) -- (2 cm,0);
      {
        \pgftransformxshift{2 cm}
        \draw (0 cm,0) -- (2 cm,0);
        \draw (2 cm, 0 cm) edge +(2 cm,0);
        \draw (4.0 cm,0) edge node [above, font=\scriptsize]{\(5\)} +(2 cm,0);
        \draw[fill=white] (0 cm, 0 cm) circle (0.4cm) node {\(0\)};
        \draw[fill=white] (2 cm, 0 cm) circle (0.4cm) node {\(1\)};
        \draw[fill=white] (4 cm, 0 cm) circle (0.4cm) node {\(0\)};
        \draw[fill=white] (6 cm, 0 cm) circle (0.4cm) node {\(0\)};
      }
      \draw[fill=white] (0 cm, 0 cm) circle (0.4cm) node {\(0\)};
    \end{tikzpicture}
    \quad
    \xmapsto{v}
    \quad
    \begin{tikzpicture}[scale=0.6, font=\scriptsize]
      \draw (0 cm,0) -- (2 cm,0);
      {
        \pgftransformxshift{2 cm}
        \draw (0 cm,0) -- (2 cm,0);
        \draw (2 cm, 0 cm) edge +(2 cm,0);
        \draw (4.0 cm,0) edge node [above, font=\scriptsize]{\(5\)} +(2 cm,0);
        \draw[fill=white] (0 cm, 0 cm) ellipse (0.8cm and 0.4cm) node {\(\tau + 1\)};
        \draw[fill=white] (2 cm, 0 cm) ellipse (0.8cm and 0.4cm) node {\(\tau + 1\)};
        \draw[fill=white] (4 cm, 0 cm) ellipse (0.8cm and 0.4cm) node {\(2\tau + 2\)};
        \draw[fill=white] (6 cm, 0 cm) ellipse (0.8cm and 0.4cm) node {\(2\tau + 1\)};
      }
      \draw[fill=white] (0 cm, 0 cm) circle (0.4cm) node {\(\tau\)};
    \end{tikzpicture}
  \end{equation}
  Consider an arbitrary element \((a,b,c,d,e)\) of the root space.
  Let us examine the application of \(\mathbf{u}\) letter by letter to this element, as follows.
  \begin{center}
    \begin{tabular}{c|c c c c c}
      \toprule
      Letter& 1 & 2 & 3 & 4 & 5\\
      \midrule
            & \(a\) & \(b\) & \(c\) & \(d\) & \(e\)\\
      \(s_3\) & \(a\) & \(b\) & \textcolor{darkred}{\(b+d-c\)} & \(d\) & \(e\)\\
      \(s_2\) & \(a\) & \textcolor{darkred}{\(a+d-c\)} & \(b+d-c\) & \(d\) & \(e\)\\
      \(s_1\) & \textcolor{darkred}{\(d-c\)} & \(a+d-c\) & \(b+d-c\) & \(d\) & \(e\)\\
      \(s_4\) & \(d-c\) & \(a+d-c\) & \(b+d-c\) & \textcolor{darkred}{\(e\tau + b - c\)} & \(e\)\\
      \(s_5\) & \(d-c\) & \(a+d-c\) & \(b+d-c\) & \(e\tau + b - c\) & \textcolor{darkred}{\((e+b-c)\tau\)}\\
      \(s_4\) & \(d-c\) & \(a+d-c\) & \(b+d-c\) & \textcolor{darkred}{\(e + d + (b-c)(1+\tau)\)} & \((e+b-c)\tau\)\\
      \(s_5\) & \(d-c\) & \(a+d-c\) & \(b+d-c\) & \(e + d + (b-c)(1+\tau)\) & \textcolor{darkred}{\(d\tau + (b-c)(1+\tau)\)}\\
      \bottomrule
    \end{tabular}
  \end{center}
  The depth strictly increases after applying each letter of \(\mathbf{u}\) if and only if the coordinate that changes at each step is strictly greater than the original value.
  We thus obtain the following inequalities.
  \begin{align}
    \label{eq:cone-inequalities}
    \begin{split}
(b + d - c) - c & > 0\\
(a + d - c) - b & > 0\\
(d - c) - a & > 0\\
(e\tau + b - c) - d & > 0\\
((e+ b - c)\tau) - e & > 0\\
(e + d + (b-c)(1+\tau)) - (e\tau + b - c) & > 0\\
      (d\tau + (b-c)(1+\tau))-((e+b-c)\tau) & > 0.
    \end{split}
  \end{align}
  Let \(C\) be the cone in the root space cut out by the inequalities above.
  By direct calculation, we can verify that
  \[v(\alpha_3) = (\tau, \tau+1, \tau+1, 2\tau + 2, 2\tau + 1)\]
lies in \(C\).
  We claim that \(u^rv(\alpha_3) \in C\), which will imply~\ref{it:c-depth}.
  Unfortunately, \(C\) itself is not \(u\)-stable, so this claim is not obvious.
  Thus we resort to an indirect and somewhat tedious argument.

  We can write the matrix of the linear transformation \(u\) on the root space as follows.
  \[u =
    \begin{pmatrix}
      0 & 0 & -1 & 1 & 0\\
      1 & 0 & -1 & 1 & 0\\
      0 & 1 & -1 & 1 & 0\\
      0 & 1 + \tau & -1-\tau & 1 & 1\\
      0 & 1+\tau & -1-\tau & \tau & 0
    \end{pmatrix}.
  \]

  This matrix has characteristic polynomial \(x^5 - \tau x^3 - \tau x^2 + 1\), which factors into five distinct factors over \(\mathbb{C}\).
  It has roots
  \[\frac{\lambda \pm \sqrt{\lambda^2 -4}}{2} \text { and } -1, \text{ where } \lambda = \frac{1 \pm \sqrt{4\tau + 5}}{2}.\]
  Thus the root space has a basis of eigenvectors.

  For \(\lambda = (1 + \sqrt{4\tau + 5})/2\), the corresponding two eigenvalues are real, and are approximately equal to
  \[\mu_1 \approx 1.5472 \text{ and } \mu_2 \approx 0.6463.\]
  For \(\lambda = (1- \sqrt{4\tau + 5})/2\), the corresponding two eigenvalues are complex conjugates of the form \[\mu_3 = e^{i\theta} \text{ and } \mu_4 = e^{-i \theta},\] where
  \(4 \cos \theta = 1 - \sqrt{4 \tau + 5}\).
  Finally, let \(\mu_5 = -1\) be the last eigenvalue.
  Thus the largest eigenvalue is \(\mu_1 \approx 1.5472\), while the other eigenvalues have modulus equal to at most \(1\).
  
  We can thus write the vector \(v(\alpha_3)\) into a basis of eigenvectors as 
  \(v(\alpha_3) = \sum_{i = 1}^5 c_i v_{\mu_i}\).
  Then for any \(n \in \mathbb{N}\) we have
  \[u^n v(\alpha_3) = \sum_{i=1}^5 \mu_i^n c_i v_{\mu_i}.\]

  We can verify by hand that \(v_{\mu_1} \in C\).
  Since \(\mu_1\) is the dominant eigenvalue, if \(c_1 \neq 0\) we must have that \(u^rv (\alpha_3) \in C\) for \(r \gg 0\).
  In fact, we desire this result for all \(r \geq 0\).
  Thus we use the following estimate.

  Consider any real-valued linear functional \(f \colon \mathbb{R}^5 \to \mathbb{R}\).
  By extending scalars, we can view this as a functional \(f \colon \mathbb{C}^5 \to \mathbb{C}\).
  Let \(x_1, x_2 \in \mathbb{C}^5\) such that \(f(x_1 + x_2 ) \in \mathbb{R}\).
  Then we have the following obvious inequality:
  \[
    f(x_1 + x_2) \geq -|f(x_1)| - |f(x_2)|.
  \]
  By iterating this observation and noting that \(v(\alpha_3)\) as well as three out of five eigenvectors \(v_{\mu_i}\) are real, we see that
  \[
    f(v(\alpha_3)) \geq f(c_1v_{\mu_1}) - \sum_{i=2}^5|f(c_iv_{\mu_i})|.
  \]
  Similarly,
  \begin{equation}\label{eq:first-inequality}
    f(u^rv(\alpha_3)) \ge \mu_1^rf(c_1 v_{\mu_1}) - \sum_{i=2}^5|\mu_i^r||f(c_iv_{\mu_i})|.
  \end{equation}
  Consider the right hand side of the inequality above.
  Dividing through by the (positive) quantity \(\mu_1^r\), we obtain
  \[f(c_1 v_{\mu_1}) - \sum_{i=2}^5\left| \left( \frac{\mu_i}{\mu_1} \right)^r\right||f(c_iv_{\mu_i})|.\]
  Since \(|\mu_i/\mu_1| < 1\) for \(i \in \{2, \ldots ,5\}\), we see in turn that
  \begin{equation}\label{eq:second-inequality}    
    f(c_1 v_{\mu_1}) - \sum_{i=2}^5\left| \left( \frac{\mu_i}{\mu_1} \right)^r\right||f(c_iv_{\mu_i})| > f(c_1 v_{\mu_1}) - \sum_{i=2}^5|f(c_iv_{\mu_i})|.
  \end{equation}
  Combining~\eqref{eq:first-inequality} with~\eqref{eq:second-inequality}, we see that
  \begin{equation}\label{eq:combined-inequality}
    f(u^rv(\alpha_3)) > \mu_1^r\left( f(c_1v_{\mu_1}) - \sum_{i=2}^5|f(c_iv_{\mu_i}) \right).
  \end{equation}
  Thus to check that \(u^rv(\alpha_3) \in C\), we simply have to check the positivity of the right hand side of~\eqref{eq:combined-inequality} for each linear functional \(f\) from~\eqref{eq:cone-inequalities}.
  We have verified this numerically via SageMath~\cite{the:25}.
  The full calculation is tedious and we do not write it out here.
  However, for the convenience of the reader and for purposes of verification, we write out the exact values of the eigenvectors we have used as well approximate values of the \(c_i\).
  
  We use the following eigenvector \(v_{\mu}\) for any eigenvalue \(\mu\):
  \[v_\mu =
    \begin{pmatrix}
      \mu^2&
      \mu^2 + \mu&
      \mu^2 + \mu + 1&
      \mu^3 + \mu^2 + \mu + 1&
      \mu^4 + \tau
    \end{pmatrix}^{t}.
  \]
  For the eigenvectors as above, the coefficients \(c_i\) for the eigenvector \(v_{\mu_i}\) are approximately as follows:
  \begin{align*}
    c_1 &\approx 0.75876, \quad c_2 \approx -0.49041,\\
    c_3 &\approx -0.07770 + 0.03905 \sqrt{-1}, \quad c_4 \approx -0.07770 - 0.03905 \sqrt{-1},\\
    c_5 &\approx -0.11294.
  \end{align*}
  Thus the claim~\ref{it:c-depth} is proved.

  We now show~\ref{it:c-min-suffix}.
  Suppose for contradiction that there is some word \(\mathbf{w}_r'\) with \(|\mathbf{w}_r'|) < r\) such that \(\mathbf{w}_r\mathbf{w}_r' \in \pal\).
  Then we can ``cancel'' \(\mathbf{w}_r'\) with the reverse of \(\mathbf{w}_r'\) from the front and back of \(\mathbf{w}_r\mathbf{w}_r'\) to yield the palindromic word
  \(\mathbf{x} \mathbf{u}^m \mathbf{v} s_3\)
  for some \(m < r\), where \(\mathbf{x}\) is a tail of \(\mathbf{u}\) of length at most \(6\).
  However, it is clear from the forms of \(\mathbf{u}\) and \(\mathbf{v}\) that no such word can be palindromic.
\end{proof}

\section{Proof of the main theorem}
The aim of this section is to prove~\cref{thm:all-coxeter-groups-main}.
We follow the strategy of the classification of finite Coxeter systems from e.g.~\cite[Theorem 2.7]{hum:90}.
The main observation is that if a Coxeter graph \(\Gamma\) has a subgraph \(\Gamma'\) whose associated Coxeter group \(W_{\Gamma'}\) is infinite, then the group \(W_{\Gamma}\) is also infinite.
In~\cref{lem:bootstrap-from-subgraph} we adapt this idea to show that if \(\Gamma\) contains one of the  ``bad subgraphs'' discussed in~\cref{sec:bad-subgraphs} then the language of palindromic reduced expressions for \(\Gamma\) is not regular.
We use this result to deduce~\cref{thm:all-coxeter-groups-main}.

We first recall the notion of a subgraph in this context.
\begin{definition}\label{def:subgraph}
  Let \(\Gamma\) be a Coxeter graph.
  A \defn{proper subgraph} \(\Gamma'\) of \(\Gamma\) is a graph obtained by doing one or both of the following:
  \begin{enumerate}
  \item omiting one or more vertices and all edges adjacent to the omitted vertices, or
  \item decreasing the labels on one or more edges.
  \end{enumerate}
  We say that \(\Gamma'\) is a \defn{subgraph} of \(\Gamma\) if either \(\Gamma' = \Gamma\) or \(\Gamma'\) is a proper subgraph of \(\Gamma\).
  If \(\Gamma'\) is a subgraph of \(\Gamma\), then we say that \(\Gamma\) \defn{contains} \(\Gamma'\).
\end{definition}

The following lemma is the key technical tool that allows us to reduce to bad subgraphs.
\begin{lemma}\label{lem:bootstrap-from-subgraph}
  Let \(\Gamma\) be a Coxeter graph and let \(\Gamma'\) be a subgraph.
  Let \((W,S)\) and \((W',S')\) be the Coxeter systems associated to \(\Gamma\) and \(\Gamma'\) respectively.
  Note that \(S' \subset S\).
  Consider a word \(\mathbf{w}\) in \((S')^{\ast}\) that is reduced in \(W'\).
  Then the same word \(\mathbf{w}\), regarded as an element of \(S^{\ast}\), is also reduced in \(W\).
\end{lemma}
\begin{remark}
  Before we prove this lemma, we make the following remark.
  Suppose that \(\Gamma'\) is a subgraph of a Coxeter graph \(\Gamma\) with associated Coxeter systems \((W',S')\) and \((W,S)\) respectively.
  There is no obvious group-theoretic relationship between \(W\) and \(W'\), and the map of monoids \((S')^{\ast} \to S^{\ast}\) described above does not necessarily descend to a group map.
  We also caution the reader that the lemma stated above is only true in one direction.
  Indeed, words in the letters \(S'\) that are reduced in the group \(W\) are not necessarily reduced in \(W'\).
  For instance, most reduced words in the infinite dihedral group become reduced in any fixed finite dihedral group.
\end{remark}
\begin{proof}[Proof of~\cref{lem:bootstrap-from-subgraph}]
  To prove this lemma, we use Matsumoto's theorem (see, e.g.~\cite[Theorem 3.3.1]{bjo.bre:05}).
  This theorem states the following.
  Suppose that \(\mathbf{w}\) and \(\mathbf{w'}\) are two words representing the same element \(w\) of a Coxeter group, such that \(|\mathbf{w'}| \leq |\mathbf{w}|\).
  Then \(\mathbf{w}\) can be transformed to \(\mathbf{w'}\) by a sequence of \emph{nil moves} (cancelling successive instances of the same letter) and \emph{braid moves} (replacing an expression of \(stst \cdots\) of length \(m(s,t)\) by the expression \(tsts\cdots\) of length \(m(s,t)\)).

  Let \(\mathbf{w'} = s_{i_1} \cdots s_{i_k}\) be an element of \((S')^{\ast}\) that is reduced in \(W'\).
  Suppose for contradiction that \(\mathbf{w}\) is not reduced in \(W\).
  Thus there is a sequence of braid and nil moves that transforms \(\mathbf{w}\) into an expression \(\mathbf{u}\) of smaller length, such that \(\mathbf{w}\) and \(\mathbf{u}\) represent the same word in \(W\).

  Consider the first move in this sequence.
  It cannot be a nil move, because otherwise the same nil move would have been possible in \(W'\) as well.
  This cannot happen because \(\mathbf{w'}\) is reduced in \(W'\).
  Thus the first move is a braid move.
  For \(s,t \in S'\), let \(m'(s,t)\) and \(m(s,t)\) be the orders of the element \(st\) in \(W'\) and \(W\) respectively.
  The first braid move is then of the form \(stst \cdots = tsts \cdots\) where each side has \(m(s,t)\) terms.
  Recall that due to the subgraph condition, we have \(m'(s,t) \leq m(s,t)\).
  If the inequality is strict, then the subexpression \(stst \cdots\) of \(\mathbf{w'}\) with \(m(s,t)\) terms is itself not reduced in \(W'\), which is impossible.
  Thus we must have \(m'(s,t) = m(s,t)\) in this case.
  This means that the first move in the sequence can be performed in both \(W\) and \(W'\).
  Proceeding along the chain of moves in this manner, we obtain a contradiction to the fact that \(\mathbf{w'}\) is reduced in \(W'\).
\end{proof}
Before proceeding with the main plan for this section, we briefly digress to discuss the following corollary of~\cref{lem:bootstrap-from-subgraph}.
\begin{proposition}\label{prop:bootstrap-cardinality}
  Let \(\Gamma\) be a Coxeter graph and let \(\Gamma'\) be a subgraph of \(\Gamma\).
  Let \((W,S)\) and \((W',S')\) be the Coxeter groups associated to \(\Gamma\) and \(\Gamma'\).
  Fix a total order on \(S'\).
  This order induces a lexicographic order on \((S')^{\ast}\).
  For any element \(w \in W'\), let \(\mathbf{w}^{\min}\) be the lexicographically minimal reduced word that represents \(w\).
  Then the function \(W' \to S^{\ast}\) defined by \(w \mapsto \mathbf{w}^{\min}\) induces a (set-theoretic) injection \(W' \to W\).
  In particular, if \(W'\) is infinite, then so is \(W\).
\end{proposition}
\begin{proof}
  Let \(\mu \colon S^{\ast} \to W\) be the natural (product) map.
  The claim is that the function \(W' \to W\) defined as \(w \mapsto \mu(\mathbf{w}^{\min})\) is injective.
  Suppose not; then there are two elements \(u,w \in W'\) such that \(\mu(\mathbf{u}^{\min}) = \mu(\mathbf{w}^{\min})\).
  Since both words are reduced in \(W'\) and hence also in \(W\) by~\cref{lem:bootstrap-from-subgraph}, they have the same number of letters.
  Then once again by Matsumoto's theorem, we can transform \(\mathbf{u}^{\min}\) to \(\mathbf{w}^{\min}\) by just a sequence of braid moves in \(W\).
  We proceed as in the proof of~\cref{lem:bootstrap-from-subgraph}.
  The first move must be a braid move of the form \(stst \cdots = tsts \cdots\) with \(m(s,t)\) factors, where \(m(s,t)\) is the order of \(st\) in \(W\).
  If \(m'(s,t)\) is the order of \(st\) in \(W'\), then we have \(m'(s,t) \leq m(s,t)\).
  If the inequality is strict, then \(\mathbf{u}^{\min}\) cannot be reduced in \(W'\).
  Thus we have \(m'(s,t) = m(s,t)\) and the first braid move is also possible in \(W'\).
  Proceeding by induction, we see that \(w = u\) in \(W'\) as well.
\end{proof}
\begin{remark}\label{rem:new-proof-classification}
  The proposition above can be used to give a slightly simpler proof of the classification of finite Coxeter groups.
  The standard proof proceeds by observing that the bilinear form~\eqref{eq:bilinear} is positive definite for any finite Coxeter system, and then showing that no other Coxeter graph can have a positive-definite bilinear form.
  We bypass the discussion of positive-definiteness, and instead use the following chain of reasoning.
  \begin{enumerate}
  \item The geometric representation is faithful (\cref{prop:geom-rep-faithful}). The proof of this statement relies on the fact that all relations in a Coxeter system are of dihedral type.
  \item The root space for all affine types, as well as the types \(Z_4\) and \(Z_5\) from the previous section, contains infinitely many roots.
    This fact is easy to verify by hand; for instance by using the ideas in the proofs of~\cref{prop:affine-types,prop:z4,prop:z5}.
  \item Any Coxeter graph that does not contain any affine type or type \(Z_4\) or \(Z_5\) as a subgraph must be one of the ones listed in the first two columns of~\cref{tab:finite-coxeter}. The proof of this follows the same ideas as in~\cite[Theorem 2.7]{hum:90}, using our~\cref{prop:bootstrap-cardinality} as a drop-in replacement for~\cite[Corollary 2.6]{hum:90}.
  \item Finally, we verify by hand that the graphs listed in first two columns of~\cref{tab:finite-coxeter} indeed have finite Coxeter groups.
    This can be done, for instance, by enumerating all the roots as in~\cite[Sections 2.10 and 2.13]{hum:90}.
  \end{enumerate}
\end{remark}
Returning to the main plan of this section, we are interested in proving that whenever a graph contains certain bad subgraphs, its associated language \(\pal\) is non-regular.
The bad subgraphs we need to handle are exactly the affine Coxeter graphs other than \(\widetilde{A_1}\), as well as \(Y_3\), \(Z_4\), and \(Z_5\).
We know from~\cref{prop:affine-types,prop:y3,prop:z4,prop:z5} that the languages of palindromic reduced expressions for these types are not regular.
The next lemma bootstraps these results to graphs containing these as subgraphs.
\begin{lemma}\label{lem:bootstrap-from-subgraph}
  Let \(\Gamma\) be a Coxeter graph that contains a subgraph that is either affine with at least two vertices, or equal to one of \(Y_3\), \(Z_4\), or \(Z_5\).
  Then the language of palindromic reduced expressions for \(\Gamma\) is not regular.
\end{lemma}
\begin{proof}
  Let \(\Gamma' \subset \Gamma\) be a subgraph of one of the types listed in the statement of the lemma.
  Let \((W',S')\) and \((W,S)\) be the corresponding Coxeter systems.
  As previously discussed, we have \(S' \subset S\).
  Let \(\pal'\) and \(\pal\) be the languages of palindromic reduced expressions for \((W',S')\) and \((W,S)\) respectively.

  Let \(\Gamma'\) be one of the subgraphs listed in the lemma: affine with at least two vertices, or one of \(Y_3\), \(Z_4\), or \(Z_5\).
  Then there are infinitely many Myhill--Nerode equivalence classes for \(\pal'\).
  Furthermore, we have produced in each case a countable list of pairwise inequivalent reduced words \(\{\mathbf{w}_r = \mathbf{u}^r\mathbf{v} s\mid r \in \mathbb{N}\}\) where \((\mathbf{u},\mathbf{v},s)\) satisfy the conditions of~\cref{lem:common-proof-schema}.
  Since \(S' \subset S\), we can regard the same words as being words in \(W\).
  When regarded as words in \(W\), we write \(\mathbf{w}_r\) as \(\mathbf{y}_r\) for clarity.

  We claim that the words \(\{\mathbf{y}_r \mid r \in \mathbb{N}\}\) contain representatives of infinitely many distinct equivalence classes with respect to the Myhill--Nerode equivalence relation of \(\pal\).
  By~\cref{lem:bootstrap-from-subgraph}, the words \(\mathbf{y}_r\) are reduced in \(W\).
  It follows from the proof of~\cref{lem:common-proof-schema} that for every \(r \in \mathbb{N}\) there is a reduced word \(\mathbf{w}_r'\) in \(\pal'\) such that \(\mathbf{w}_r\mathbf{w}_r' \in \pal'\).
  Take any such \(\mathbf{w}_r'\) and consider the corresponding word \(\mathbf{y}'_r\) in \(W\).
  Then \(\mathbf{y}_r\mathbf{y}'_r\) is reduced by~\cref{lem:bootstrap-from-subgraph}, and evidently palindromic, and thus \(\mathbf{y}_r\mathbf{y}_r' \in \pal\).
  The same proof strategy as that of~\ref{it:c-min-suffix} in each case shows that \(|\mathbf{y}'_r| > r\).

  Now by repeating the proof strategy of~\cref{lem:common-proof-schema}, we see that the set \(\{\mathbf{y}_r \mid r \in \mathbb{N}\}\) contains infinitely many Myhill--Nerode equivalence classes for \(\pal\).
  Thus \(\pal\) is not regular.
\end{proof}

The next result classifies all Coxeter graphs that contain the ``bad subgraphs'' from~\cref{prop:affine-types,prop:y3,prop:z4,prop:z5}.
It is almost a restatement of the classification theorem of finite and affine Coxeter groups, which is e.g.~\cite[Theorem 2.7]{hum:90}.
\begin{proposition}\label{prop:bad-subgraphs}
  Let \(\Gamma\) be a connected Coxeter graph and let \(W_{\Gamma}\) be its associated Coxeter group.
  If \(\Gamma\) is not of type \(\widetilde{A_1}\), then \(W_{\Gamma}\) is infinite if and only if it contains one of the affine Coxeter graphs, or one of the graphs \(Y_3\), \(Z_4\), or \(Z_5\).
\end{proposition}
\begin{proof}
  This result follows from the proof of~\cite[Theorem 2.7]{hum:90} after a small modification.
  That proof consists of a 20-step flowchart that classifies the possible \(\Gamma\).
  The aim of that proof is to rule out all types other than affine types.
  For our proof, we wish to rule out all affine types other than \(\widetilde{A_1}\).

  Let \(\Gamma\) be a connected Coxeter graph that is not \(\widetilde{A_1}\), and does not contain as a subgraph one of the affine types with at least \(2\) vertices, or any of the graphs \(Y_3\), \(Z_4\), or \(Z_5\).
  Let \(n\) be the number of vertices in \(\Gamma\), and let \(m\) be its maximum edge label.

  If \(n = 2\) then \(\Gamma\) must be of finite dihedral type.
  So let us assume that \(n \geq 3\).
  \(\Gamma\) cannot have a cycle, otherwise it would contain affine \(A_n\) for some \(n\).
  Thus \(\Gamma\) is a tree.
  We must have \(m < \infty\), otherwise \(\Gamma\) contains \(Y_3\) as a subgraph.
  After this modification, the remainder of the proof of~\cite[Theorem 2.7]{hum:90} (starting from step (3) of that proof) goes through verbatim.
\end{proof}

The next lemma handles Coxeter systems that may have disconnected Coxeter graphs.
\begin{lemma}\label{prop:disconnected-coxeter-system}
  Let \(\Gamma\) be any Coxeter graph with connected components \(\Gamma_1, \ldots, \Gamma_r\).
  Let \(\pal_1, \ldots, \pal_r\) be the languages of palindromic reduced expressions for \(\Gamma_1, \ldots, \Gamma_r\), and let \(\pal\) be the language of palindromic reduced expressions for \(\Gamma\).
  Then \(\pal\) is regular if and only if each \(\pal_i\) is regular.
\end{lemma}
\begin{proof}
  Let \((W,S)\) be the Coxeter system corresponding to \(\Gamma\), and let \((W_i, S_i)\) be the Coxeter system corresponding to each \(\Gamma_i\).
  Then \(S\) is the disjoint union of the \(S_i\).
  
  We first show that \(\pal\) is simply the union of the languages \(\pal_i\).
  Indeed, consider any palindromic reduced expression in \(\pal\), of the form \(\mathbf{w}\, s
  \,\overline{\mathbf{w}} \neq 1\).
  Since \(s \in S\) and \(S\) is the disjoint union of all \(S_j\), there is a unique \(i\) such that \(s \in S_i\).
  We claim that \(\mathbf{w}\) cannot contain any letter of \(S_j\) for \(j \neq i\).
  Indeed, suppose for contradiction that \(\mathbf{w}\) contains a letter of \(S_j\) for some \(j \neq i\).
  Let \(a\) be the rightmost such letter in \(\mathbf{w}\).
  That is, \(\mathbf{w} = \mathbf{u} \,a \,\mathbf{v}\) for some \(a \in S_j\) with \(j \neq i\).
  Moreover, \(\mathbf{u}\) and \(\mathbf{v}\) are words such that \(\mathbf{v}\) only contains letters of \(S_i\).
  Then \(w^{-1} = v^{-1}a u^{-1}\).
  Since \(a\) commutes with all elements of \(S_i\), we have \[w s w^{-1} = u a v s v^{-1}a u^{-1} = u v s v^{-1}(aa)u^{-1} = u v s v^{-1}u^{-1}.\]
  However, since we assumed that \(\mathbf{w}\, s\, \overline{\mathbf{w}}\) is reduced, no such cancellation can occur.
  Therefore all the letters in \(\mathbf{w} \,s\, \overline{\mathbf{w}}\) belong to \(S_i\).
  This argument proves that \(\pal = \bigcup_i \pal_i\).

  Suppose that each language \(\pal_i\) is regular.
  Since any finite union of regular languages is regular, \(\pal\) is also regular.
  Thus one direction of the proposition is proved.

  On the other hand, the language \(\pal_i\) is simply the intersection of \(\pal\) with the regular language \(S_i^{\ast}\).
  Since the intersection of two regular languages is regular, we see that if \(\pal\) is regular, then so is \(\pal_i\) for each \(i\).
  Thus the other direction is also proved.
\end{proof}

We can now finish the proof of the main theorem.
\begin{proof}[Proof of~\cref{thm:all-coxeter-groups-main}]
  Suppose that \((W,S)\) is the Coxeter system associated to a Coxeter graph \(\Gamma\).
  Suppose that \(\Gamma\) has connected components \(\{\Gamma_i \mid 1 \leq i \leq r\}\) with associated Coxeter systems \(\{(W_i,S_i) \mid 1 \leq i \leq r\}\).
  Let \(\pal\) (resp.~\(\pal_i\)) be the language of palindromic reduced expressions for \(W\) (resp.~each \(W_i\)).
  
  If each \(W_i\) is either finite or of type \(\widetilde{A_1}\), then each \(\pal_i\) is regular (see~\cref{prop:affine-a1}).
  By~\cref{prop:disconnected-coxeter-system}, the language \(\pal\) is regular.
  Thus one direction of the theorem is proved.

  To prove the other direction, suppose that there is some \(i\) for which \(W_i\) is neither finite nor of type \(\widetilde{A_1}\). 
  By~\cref{prop:bad-subgraphs}, the Coxeter graph \(\Gamma_i\) contains one of the affine Coxeter graphs with at least three vertices, or one of the types \(Y_3\), \(Z_4\), or \(Z_5\).
  By~\cref{prop:affine-types,prop:y3,prop:z4,prop:z5}, the languages of palindromic reduced expressions of any of these subgraphs are not regular.
  By~\cref{lem:bootstrap-from-subgraph}, \(\pal_i\) is also not regular.
  Finally, by~\cref{prop:disconnected-coxeter-system}, \(\pal\) is not regular.
  The proof is complete.
\end{proof}

\bibliographystyle{siam}

\begin{thebibliography}{10}

\bibitem{bia.hoh.sas:26}
{\sc R.~Biagioli, C.~Hohlweg, and E.~Sasso}, {\em On generating functions and
  automata associated to reflections in {Coxeter} systems},  (2026).
\newblock \url{https://arxiv.org/abs/2602.16361}.

\bibitem{bjo.bre:05}
{\sc A.~Bj{\"o}rner and F.~Brenti}, {\em Combinatorics of {Coxeter} groups},
  vol.~231 of Grad. Texts Math., New York, NY: Springer, 2005.

\bibitem{bre:24}
{\sc F.~Brenti}, {\em Some open problems on {Coxeter} groups and unimodality},
  (2024).
\newblock \url{https://arxiv.org/abs/2410.09897}.

\bibitem{bri.how:93}
{\sc B.~Brink and R.~B. Howlett}, {\em A finiteness property and an automatic
  structure for {Coxeter} groups}, Math. Ann., 296 (1993), pp.~179--190.

\bibitem{man:99}
{\sc R.~de~Man}, {\em The generating function for the number of roots of a
  {Coxeter} group}, J. Symb. Comput., 27 (1999), pp.~535--541.

\bibitem{dye.fis.hoh.ea:24}
{\sc M.~Dyer, S.~Fishel, C.~Hohlweg, and A.~Mark}, {\em Shi arrangements and
  low elements in {Coxeter} groups}, Proc. Lond. Math. Soc. (3), 129 (2024),
  p.~48.
\newblock Id/No e12624.

\bibitem{dye:87}
{\sc M.~J. Dyer}, {\em Hecke algebras and reflections in {Coxeter} groups}, PhD
  thesis, The University of Sydney, 1987.
\newblock \url{https://academicweb.nd.edu/~dyer/papers/dyerthesis.pdf}.

\bibitem{dye:26}
\leavevmode\vrule height 2pt depth -1.6pt width 23pt.
\newblock Personal communication, 2026.

\bibitem{edg:09}
{\sc T.~Edgar}, {\em {Dominance and regularity in {Coxeter} groups}}, PhD
  thesis, University of Notre Dame, 2009.
\newblock \url{https://doi.org/10.7274/wm117m03w68}.

\bibitem{hum:90}
{\sc J.~E. Humphreys}, {\em Reflection groups and {C}oxeter groups}, vol.~29 of
  Cambridge Studies in Advanced Mathematics, Cambridge University Press,
  Cambridge, 1990.

\bibitem{mil:25}
{\sc E.~Mili{\'c}evi{\'c}}, {\em Reduced words for reflections in {Weyl}
  groups}, Result. Math., 80 (2025), p.~24.
\newblock Id/No 100.

\bibitem{sip:06}
{\sc M.~Sipser}, {\em Introduction to the theory of computation}, Boston, MA:
  Thompson, 2nd~ed., 2006.

\bibitem{the:25}
{\sc {The Sage Developers}}, {\em {S}ageMath, the {S}age {M}athematics
  {S}oftware {S}ystem ({V}ersion 10.7)}, 2025.
\newblock \url{https://www.sagemath.org}.

\end{thebibliography}

\end{document}